\pdfoutput=1
\documentclass[11pt, a4paper]{scrartcl}  
\usepackage[l2tabu, orthodox]{nag}  
\usepackage[utf8]{inputenc}
\usepackage[pdfusetitle]{hyperref}

\usepackage{amsmath, amsthm, amssymb}

\usepackage{mathptmx}
\usepackage[scaled=.90]{helvet}
\usepackage{courier}

\usepackage{authblk}

\usepackage[semicolon, round]{natbib}
\usepackage[usenames]{color}
\usepackage{cleveref}

\usepackage{graphicx}
\graphicspath{{./figures/}}

\theoremstyle{plain}
\newtheorem{theorem}{Theorem}
\newtheorem{lemma}{Lemma}

\theoremstyle{definition}
\newtheorem{definition}{Definition}

\numberwithin{equation}{section}

\DeclareMathOperator{\sech}{sech}

\newcommand{\dd}[2]{\frac{\mathrm{d} #1}{\mathrm{d} #2}}

\begin{document}

\author[1,*]{Chris~J.~Budd}
\author[1,2]{Andrew~T.~T.~McRae}
\author[3]{Colin~J.~Cotter}
\affil[1]{Department of Mathematical Sciences, University of Bath, Bath, BA2 7AY, UK}
\affil[2]{Atmospheric, Oceanic and Planetary Physics, University of Oxford, Oxford, OX1 3PU, UK}
\affil[3]{Department of Mathematics, Imperial College London, London, SW7 2AZ, UK}
\affil[*]{Correspondence to: \texttt{c.j.budd@bath.ac.uk}}
\title{The scaling and skewness of optimally transported meshes on the sphere}
\date{}
\maketitle

\begin{abstract}
  In the context of numerical solution of PDEs, dynamic mesh redistribution
  methods (r-adaptive methods) are an important procedure for increasing
  the resolution in regions of interest, without modifying the
  connectivity of the mesh. Key to the success of these methods is that
  the mesh should be sufficiently refined (locally) and flexible in
  order to resolve evolving solution features, but at the same time not
  introduce errors through skewness and lack of regularity. Some
  state-of-the-art methods are bottom-up in that they attempt to
  prescribe both the local cell size and the alignment to features of
  the solution. However, the resulting problem is overdetermined,
  necessitating a compromise between these conflicting requirements.
  An alternative approach, described in this paper, is to prescribe only
  the local cell size and augment this an optimal transport condition to
  provide global regularity. This leads to a robust and flexible
  algorithm for generating meshes fitted to an evolving solution, with minimal
  need for tuning parameters. Of particular interest for geophysical
  modelling are meshes constructed on the surface of the sphere. The
  purpose of this paper is to demonstrate that meshes generated on the
  sphere using this optimal transport approach have good a-priori 
  regularity and that the meshes produced are naturally aligned to
  various simple features. It is further shown that the sphere's
  intrinsic curvature leads to more regular meshes than the plane. In
  addition to these general results, we provide a wide range of examples
  relevant to practical applications, to showcase the behaviour of
  optimally transported meshes on the sphere. These range from
  axisymmetric cases that can be solved analytically to more general
  examples that are tackled numerically. Evaluation of the singular
  values and singular vectors of the mesh transformation provides a
  quantitative measure of the mesh anisotropy, and this is shown to
  match analytic predictions.
\end{abstract}

\textbf{Keywords:} Mesh adaptation, mesh regularity, optimal transport

\section{Introduction}
\label{sec:Intro}

\subsection{Overview}

Many partial differential equations are naturally formulated on the
sphere, $S^2$, or on a thin spherical shell. A notable example is those
equations describing atmospheric or oceanic flows on the Earth, which
are essential in weather forecasting and climate predictions. To find
approximate solutions to these, it is common to first define a mesh on
the sphere (perhaps extended in vertical columns in the case of a thin
shell). The equations are then discretized with respect to this mesh,
using, for example, a finite difference or finite volume method.

There are many considerations for constructing a suitable mesh. Firstly,
the mesh must be a reasonable approximation of the analytic domain. It
is important that the solution can be faithfully represented on the
mesh; this may be non-trivial if the solution develops small-scale
features that evolve and move around over time. In
\citet{slingo2009developing}, it is highlighted that a major limitation
of the development of climate models is the lack of mesh resolution when
faced with climatic phenomena on many scales. The resulting need for
using some form of mesh refinement to resolve important atmospheric
features is emphasised. Issues can arise from interactions between the
mesh and the numerical method being used, as described in
\citet{staniforth2012horizontal}, where the desirability of particular
degree-of-freedom ratios between different fields (such as air pressure
and wind speed) leads to constraints on the topology of the mesh.
Furthermore, if computational efficiency is important, a structured mesh
is desirable. These allow direct addressing to be used, and so generally
lead to faster calculations than unstructured meshes -- although the
difference can be minimised in the spherical shell case since the radial
direction always provides exploitable structure
\citep{macdonald2010general,bercea2016structure}.

Crucial to the success of such methods is the development of algorithms
which (with minimal operator intervention) can generate a mesh rapidly.
This mesh must be able to resolve the small scale features of the
underlying solution, align itself to anisotropic solution features, and
yet be sufficiently regular to avoid errors due to excessive mesh
skewness. One method for doing this is r-adaptive mesh relocation, in
which a fixed number $N$ of mesh points, with prescribed connectivity,
is moved around the sphere so that points are concentrated in regions
requiring higher mesh resolution. In earlier papers
\citep{weller2016mesh,mcrae2018optimal}, we have demonstrated that
r-adaptive mesh relocation using optimal transport regularisation can be
implemented, flexibly and rapidly, both in the plane and on the surface
of the sphere. These meshes have been shown to avoid tangling and to
be capable of following time-evolving features with excellent resolution
of small scales. The unchanging mesh topology in the r-adaptive approach
also means that all data structures remain constant throughout the
simulation; these data structures can be based on well-established
static mesh topologies such as cubed-sphere or icosahedral meshes.

The purpose of this paper is to provide novel analytical estimates for
the regularity of such meshes on the sphere by exploiting the structure
of the Monge--Ampère equation. We will show that, in general, meshes on
the sphere have better regularity than those on the plane. The general
theory will be illustrated by firstly looking at some specific examples
relevant to practical applications, where we have exact analytic
descriptions of optimally transported meshes, and then looking at some
numerically-generated meshes for more challenging examples.
Comparisons will be made with other methods for mesh generation on the
sphere. We aim to convince the reader that the ease and flexibility of
use of the optimal transport methods, combined with their good
regularity properties, make them very suitable for calculating rapidly
evolving PDEs defined on the surface of the sphere, such as those used
in geophysical modelling.

\subsection{Some existing mesh generation methods and their properties}

We can broadly define three approaches to mesh generation for transient
simulations. In the first approach, the mesh is (reasonably) uniform and
is \emph{static} -- the mesh does not evolve with the simulation.
In the second approach, the mesh is again static, but now non-uniform,
with user-prescribed local resolution. These meshes are often
unstructured. A typical example is an ocean simulation in which the
coastline is resolved. In the third approach, the mesh is
\emph{dynamic}, changing as the solution evolves. This case further
separates into mesh refinement methods and mesh relocation methods
(r-adaptive methods), which we discuss in this paper.

The meteorological community have traditionally used a
latitude--longitude mesh of the sphere, in which the sphere is divided
into cells by lines of constant latitude and constant longitude. This
mesh has the advantages of being fully structured and of having
quadrilateral cells and orthogonal gridlines, which are beneficial for
certain numerical schemes. However, the severe resolution clustering at
the poles is problematic for numerical methods and for parallel
efficiency. As a response to this problem, more uniform meshes have been
considered, particularly in recent years
\citep{williamson2007evolution,staniforth2012horizontal}. Numerous
`next-generation' models use varieties of cubed-sphere or icosahedral
meshes \citep{ullrich2017dcmip} (the spectral community has a longer
history of using `reduced' grids \citep{hortal1991reduced}). Both
cubed-sphere and icosahedral meshes can be represented as a collection
of structured patches. The standard varieties of these meshes are fairly
uniform, with cell areas that vary by at most a factor of two. These
are, therefore, not immediately suitable for resolving evolving
small-scale features in the solution. However, we can generalise to
using meshes that are similar to these, but which also make use of
r-adaptive mesh redistribution or a local refinement strategy to
increase the resolution in particular regions. For atmospheric flows,
this may be appropriate when the solutions of the governing PDEs develop
small-scale features in particular locations, such as atmospheric fronts
or equatorial waves.

Dynamic mesh adaptivity on the sphere is used in
\citet{mccorquodale2015adaptive} and \citet{ferguson2016analyzing} for
a shallow-water model, with impressive results. These use local mesh
refinement, rather than r-adaptivity, on an underlying cubed-sphere
grid. By exploiting the structure of the cubed-sphere grid -- each of
the six panels is locally cartesian -- the resulting adaptivity problem
is similar to adaptive mesh refinement on the plane, with extra care
needed where the panels join. Existing mesh relocation methods on the
sphere appear to be aimed at static, rather than dynamic, adaptivity. An
approach based on spring dynamics is used to generate meshes for the
NICAM dynamical core \citep{satoh2014nonhydrostatic}, although this is
motivated by producing maximally-uniform meshes -- arguably the opposite
of adaptivity! A prescribed mesh is projected onto the sphere. The mesh
is then smoothed by considering the vertices to be connected by springs
and allowing the mesh to relax to a minimal-energy configuration. The
paper \citet{tomita2002optimization} attempts to iron out the small
non-uniformities in refined icosahedral grids by varying the natural
spring length that is used. The later paper \citet{iga2014improved}
shows that a natural spring length of zero is a natural choice from
energy considerations. However, this leads to mesh distortion around the
12 five-neighbour points in the mesh, and so a custom transformation is
introduced in the neighbourhood of these points. The more recent papers
\citet{iga2015smooth,iga2017equatorially} deliberately generate grids
with higher resolution near the equator. This is done by introducing
meshes with more cells near the equator (see Figure 1 in both papers),
rather than basic icosahedral grids. The paper
\citet{ringler2011exploring} uses spherical centroidal Voronoi
tessellation, with a spatially-varying mass distribution, to produce a
static mesh with increased resolution in one part of the world. The
method is not strictly r-adaptivity: the number of mesh points remains
the same, but the connectivity is allowed to vary in order to reduce
cell stretching.

These mesh relocation methods are not particularly well-suited to
dynamic adaptivity. The NICAM grids are very pleasing, but their
generation required significant human intervention and the tuning of
free parameters. This is reasonable for static adaptivity, but the
methods do not easily generalise to producing meshes that can resolve
arbitrary time-varying features. A distinct approach to introducing mesh
adaptivity using spring dynamics would be to vary the spring constants
and/or the natural lengths. It is feasible that this process could be
performed each timestep. However, the relationship between the spring
parameters and the resulting mesh resolution is complicated -- it is far
from clear what spring parameters are needed in order to produce the
desired resolution distribution. The centroidal Voronoi approach
described in \citet{ringler2011exploring} relies on Lloyd's algorithm,
which is extremely expensive. As a result, this approach is likely
infeasible for dynamic adaptivity. Furthermore, the changing
unstructured connectivity would be hard to couple to a PDE solver.

An alternative approach for constructing an r-adaptive mesh is to
explicitly prescribe the local scale of the mesh, via a
(solution-dependent) mesh density monitor function, while imposing
global regularity in the form of an optimal transport constraint. This
has been implemented and analysed on the plane
\citep{budd2006parabolic,budd2009moving,chacon2011robust,browne2014fast,budd2015geometry},
and in recent papers, we have extended this to produce solution-adapted
meshes on the sphere \citep{weller2016mesh,mcrae2018optimal}.
These methods used an optimal
transport approach linked to the solution of a (version of the)
Monge--Ampère equation, posed on the tangent bundle to the sphere, to
produce regular-looking meshes with the desired (spatially-varying)
density of mesh points. The r-adaptive approach, in which a fixed number
of mesh points with constant connectivity is moved around, is
well-suited to PDE computations as it allows the use of fixed and simple
data structures, and requires no modification on parallel architectures.
The papers \citet{weller2016mesh,mcrae2018optimal} were primarily
computational. In this companion paper, we look at the geometry of the
resulting meshes generated using the optimal transport methods,
particularly local mesh scaling, local mesh regularity, and mesh
alignment. We demonstrate that these methods have the merit of being a
systematic approach to mesh generation that delivers meshes of
prescribed local scale and of provable regularity, through the use of
robust algorithms which are relatively simple to implement. They also
have significant flexibility in the control of the mesh points. Optimal
transport not only gives a cheap, reliable, robust and flexible means of
generating a (solution-dependent) mesh, but also has provable regularity
bounds. Furthermore, we show that there are a number of exact solutions
of the Monge--Ampère equation on the sphere which generate meshes that
are appropriate for the solution of certain meteorological problems
described in, for example, \citet{slingo2009developing}.

\subsection{Summary of this paper}

The remainder of this paper is structured as follows. In
\cref{sec:prelim}, we introduce some basic theory of r-adaptive mesh
relocation strategies, based on controlling the density of mesh points.
In particular, we consider meshes which equidistribute a monitor
function, $m$, of the mesh density. We then describe some measures of
mesh quality for meshes obtained by relocation strategies. In
\cref{sec:otsphere}, we consider the construction of meshes on the
sphere, and possibly more general two-dimensional Riemannian manifolds.
This is achieved through the use of optimal transport maps from the
sphere to itself, which act on a base mesh of the sphere. The maps are
obtained by solving scalar partial differential equations of
Monge--Ampère type. We then make significant use of the structure
provided by the optimal transport formulation to derive some a priori
estimates of mesh regularity. These results follow directly from the
analytic theory of optimally transported maps. We draw the surprising
conclusion that, due to the positive intrinsic curvature of the sphere,
meshes on the sphere tend to have better formal regularity than
analogous meshes on the plane. In \cref{sec:axisym}, we consider the
specific case of monitor functions, $m$, which are axisymmetric about a
particular axis. By solving the equidistribution equation exactly in
this case, we derive analytic expressions for the related optimally
transported meshes, and the associated mesh quality measures. In
\cref{sec:numaxi}, we then apply the optimal transport maps, with
particular monitor functions, to some common meshes of the sphere,
including latitude--longitude, cubed-sphere and icosahedral meshes,
to construct examples of meshes of proven regularity. In
\cref{sec:planecomp}, we briefly compare a spherical example to a
`matching' planar example, and show that the geometry of the sphere
leads to a mesh of higher quality. In \cref{sec:numnonaxi}, we provide
more general computational examples on the sphere, looking at some
challenging examples. In each case, we compute the regularity of the
mesh and also show that the resulting meshes are naturally aligned to
the prescribed features. Some comparisons are made with similar meshes
obtained using different mesh construction algorithms such as the spring
dynamics method. Finally, in \cref{sec:conc}, we draw some conclusions
and consider future work and applications of these methods.

\section{Mesh construction through mesh point relocation, and measures of mesh regularity.}
\label{sec:prelim}

\subsection{Mesh construction}
In this paper, we consider an r-adaptive approach to mesh construction.
As described earlier, this involves a fixed number of mesh points being
\emph{relocated}, while the topology of the mesh remains unchanged.
Suppose that we wish to construct a mesh $\tau_P$ for simulating a
physical problem in a domain $\Omega_P$, where $\Omega_P$ lies in a
manifold $M$. We assume that $\tau_P$ will be specially adapted for the
problem and may be highly non-uniform. In this work, we are particularly
interested in the case where $M$ is the surface of the sphere, $S^2$,
with $\Omega_P = M$. However, we will frequently draw comparisons to the
planar case $M = \mathbb{R}^2$. Previous work on this planar case can be
found in \citet{budd2015geometry}.

We also define a `computational' (or `logical') domain
$\Omega_C \subseteq M$, with a computational mesh $\tau_C$ of prescribed
connectivity. We assume the cells of $\tau_C$ are reasonably
uniform---perhaps fully uniform in the planar case---having
shapes and sizes that do not vary too much. This is the case in an
icosahedral mesh on the sphere refined through repeated bisection, and
in a gnomonic cubed-sphere mesh. In the \emph{r}-adaptive approach, we
assume the existence of a bijective map $F: \Omega_C \to \Omega_P$, with
$\tau_P$ the image of $\tau_C$ under the action of this map. It follows
that $\tau_P$ will have the same topology (connectivity) as $\tau_C$.

We use $\vec{\xi}$ to denote a position vector in $\Omega_C$, and
$\vec{x}$ to denote the corresponding position in $\Omega_P$:
$F(\vec{\xi}) = \vec{x}$. Let $U$ be a small open set containing
$\vec{\xi}$, and let $V$ be the image of this set under the action of
$F$ (hence containing $\vec{x}$). We may compute the ratio of the
volumes (areas) of these two sets, $|V|/|U|$. In the limit $|U| \to 0$,
we define
\begin{equation}
  \label{eq:UVratio}
  r(\vec{\xi}) = \lim_{|U| \to 0} \frac{|V|}{|U|}
\end{equation}
to be the limiting area ratio. If $M = \mathbb{R}^2$, we have
\begin{equation}
  r(\vec{\xi}) = |\det J|,
\end{equation}
where $J$ is the Jacobian of the map $F$. If $M$ is a general Riemannian
manifold of dimension $d$ embedded in some $\mathbb{R}^n$, $r$ is now
the product of the first $d$ singular values of $J$. This coincides with
the \emph{pseudodeterminant} of $J$---the product of the non-zero
singular values---as long as $F$ is not degenerate.

In \emph{r}-adaptive methods, controlling this area ratio is always a
primary concern. An \emph{equidistribution} principle is widely used:
let $m(\vec{x})$ be a suitable \emph{monitor function}, traditionally
related to the error in representing the solution on the physical mesh.
We then seek a mesh where the area ratio is inversely proportional to
$m$:
\begin{equation}
  \label{eq:equi}
  m(\vec{x}) r(\vec{\xi}) = \alpha,
\end{equation}
where $\alpha$ is a normalisation constant that ensures the domains
$\Omega_C$ and $\Omega_P$ have the correct size (alternatively, we could
impose a condition on $m$ to ensure the correct scaling). The error in
representing the solution would then be \emph{equidistributed} between
cells of the physical mesh. The monitor function $m$ does not have to be
a proxy for interpolation error; a far more general monitor function can
be used (for example, \citet{weller2016mesh} shows a mesh of the Earth
adapted to the amount of precipitation that fell on a particular day).
By \cref{eq:equi}, if $m(\vec{x})$ is large in a region, the cells of
the physical mesh $\tau_P$ are small there (as $r(\vec{\xi})$ is forced
to be small in the preimage of this region). This is desirable if higher
resolution is sought in that area.

We refer to \cref{eq:equi} as the \emph{equidistribution condition}. It
is clear that it does not, on its own, lead to a well-posed mesh
generation problem (other than in one-dimension), since the resulting
map is far from unique. It is necessary to augment the equidistribution
condition with further conditions. Traditionally, these have been
constraints on mesh regularity such as orthogonality
\citep{thompson1998handbook} or alignment to a prescribed tensor field
\citep{huang2011adaptive}. In the latter approach, the resulting mesh is
then chosen to minimise some weighted sum of terms, each attempting to
enforce a separate condition. In particular, the mesh is not designed to
satisfy \cref{eq:equi} exactly.

An alternative, and powerful, technique is to use the concept of
\emph{optimal transport} \citep{villani2003topics,villani2009optimal};
previous work using this approach includes \citet{budd2006parabolic,
budd2009moving,chacon2011robust,browne2014fast}. We now seek a map $F$
satisfying \cref{eq:equi} exactly (up to discretisation error, at least)
so that the resulting mesh $\tau_P$ is ``as close as possible" to
$\tau_C$. This ``distance" between the meshes is defined as
\begin{equation}
  \int_{\Omega_C} \|\vec{x}(\vec{\xi}) - \vec{\xi}\|^2 \,\mathrm{d}\vec{\xi},
\end{equation}
the integral of the squared Riemannian distance. In optimal transport
terminology, this is the cost of a candidate map $F: \Omega_C \to \Omega_P$.

It is well-known that a unique solution exists for this problem; this
was established in \citet{brenier1991polar} for Euclidean space and in
\citet{mccann2001polar} for the sphere. In Euclidean space, the
appropriate map can be written in the form
\begin{equation}
  \label{eq:xiplusgradu}
  \vec{x} = F(\vec{\xi}) = \vec{\xi} + \nabla u,\quad
\end{equation}
for a suitable scalar `potential' $u$. The corresponding area ratio is
\begin{equation}
  \label{eq:detIplusH}
  r(\vec{\xi}) = \det(I + H(u)),
\end{equation}
where $H(u)$ denotes the \emph{Hessian} of $u$. In \cref{sec:otsphere},
we extend this to the sphere. Combining \cref{eq:equi} with
\cref{eq:detIplusH} then leads to a nonlinear partial differential
equation for $u$, a \emph{Monge--Ampère} equation. Such equations,
defined over general manifolds $M$, have been well-studied. Many results
are available on the formal regularity of the solutions, including
estimates for the various derivatives of the function $u$ in terms of
$m$ \citep{trudinger1984second,caffarelli1990interior,
wang1995counterexamples,gutierrez2001monge,delanoe2006gradient,
caffarelli2008regularity,loeper2009regularity,loeper2011regularity}. The
formal regularity properties of the resulting mesh $\tau_P$ can then be
determined from these estimates.

In a transient simulation, the monitor function will vary with time:
$m = m(\vec{x}, t)$. In contrast to some other r-adaptive techniques
(particularly \emph{velocity-based} moving mesh methods), our approach
to mesh generation is effectively quasi-static. The mapping
$\vec{x}(\vec{\xi}, t)$ is uniquely defined by the monitor function
$m(\vec{x}, t)$ at that moment; our meshes \emph{do not know their history}.
For this reason, the effect of time does not enter our analysis.

\subsection{Geometrical measures of mesh regularity}

A general mesh $\tau$ defined on a two-dimensional Riemannian manifold
$M$ is comprised of a set of nodes on $M$ connected together by edges,
defining a set of cells. In a well-behaved mesh, the nodes are regularly
spaced and the edges meet at carefully-controlled angles. The resulting
cells are therefore not too skew. The spherical meshes mentioned
previously have these desirable properties, as does the uniform mesh on
the plane. In an \emph{r}-adaptive context, such meshes are appropriate
for a computational mesh $\tau_C$. Under the action of the map $F$, the
physical mesh $\tau_P$ is typically less regular than $\tau_C$. In
particular, the equidistribution condition \cref{eq:equi} controls the
size of the mesh cells; if this varies, it leads to a degree of skewness
in the mesh. A general criticism of \emph{r}-adaptive meshes is that
they can lead to excessively skew meshes. It can, however, be shown that
the use of the optimal transport regularisation condition (which forces
$\tau_P$ to be ``as close as possible" to $\tau_C$) is beneficial for
controlling the degree of skewness \citep{delzanno2008optimal}.

Assume that we are constructing meshes on a two-dimensional manifold.
We can then define the \emph{local scaling} and \emph{local skewness} in
terms of the linearisation of the map $F$.
\begin{definition}
Let $F: \Omega_C \to \Omega_P$ have Jacobian $J$, and suppose that this
linear operator has leading singular values $\sigma_1, \sigma_2$.

The \emph{\textbf{local scaling}} $s$ is defined as
\begin{equation}
  \label{eq:locscale}
  s = \sigma_1 \sigma_2.
\end{equation}

The \emph{\textbf{local skewness}} $Q$ is defined as
\begin{equation}
  \label{eq:locskew}
  Q = \frac{1}{2} \left(\frac{\sigma_1}{\sigma_2} + \frac{\sigma_2}{\sigma_1}\right).
\end{equation}
\end{definition}

In \cref{thm:axireg}, we will obtain estimates for both of these
quantities for a certain class of meshes induced by axisymmetric monitor
functions.

This two-dimensional local skewness measure, $Q$, is equivalent to the
mesh quality measure $Q_{geo}$ defined in \citet[p.~205]{huang2011adaptive},
where more analysis is given. In Section 5.1 of that book, it is shown
how the interpolation error for a mesh, using different types of
interpolant, can be calculated directly in terms of the scaling and
skewness. Both need to be controlled to get a low overall error, but it
may well be that one can be large provided that the other is small.

If $F$ is the identity, so that the physical mesh is equal to the
computational mesh, then $s$ is constant and $Q = 1$. This is optimal
for $Q$, but may be suboptimal for $s$ if the underlying solution we are
trying to represent on the mesh has very small scales. For an adapted
mesh it is expected that both will vary. The local scaling $s$ is
controlled directly via the equidistribution condition \cref{eq:equi}.
In contrast, the local skewness $Q$ follows indirectly from properties
of the Monge--Ampère equation. We will consider using the local skewness
as a general measure of the quality of the mesh.

These general considerations lead to the following challenges for
adaptive mesh generation:

\textbf{Challenge 1} Derive a mesh which has minimal skewness $Q$ for a
given scaling distribution $s$.

\textbf{Challenge 2} Derive a mesh which leads to minimal solution
error, expressed as a combination of scale, skewness and other factors
such as alignment properties \citep{huang2011adaptive}.

Both questions are very hard to answer in general. However we will show
that the use of the optimal transport regularisation gives some partial
answers in terms of mesh generation, and that we can give a much more
complete analysis in the case of axisymmetric monitor functions.

We remark that on the plane $\mathbb{R}^2$, the use of optimal transport
techniques leads to $J$ being symmetric, and so the expressions above in
\cref{eq:locscale,eq:locskew} can be written with eigenvalues replacing
singular values. However, this is not true for a general manifold
embedded in Euclidean space.

Our measure of skewness above is based only on the local linearisation
of the map $F$. There are many other measures of mesh quality that are
calculated directly from the vertices and edges of the mesh, such as
those used in \citet{weller2016mesh}. On the other hand, our skewness
measure doesn't formally provide this geometric information; rather it
is a function of the map $F: \Omega_C \to \Omega_P$. The geometric
information of $\tau_P$ is only recovered when the computational mesh
$\tau_C$ is specified.

For example, if $\tau_C$ is made up of small square elements, these will
be mapped to small quadrilaterals, whose precise shape depends on how
the original square is aligned to the (orthogonal) eigenvectors/singular
vectors of $F'$. If we consider the ratio, $r$, of the lengths of the
two sides of the quadrilateral to be a measure of its skewness, $Q$ is
roughly the mean of $r$ taken over all possible alignments.

Alternative regularity measures, derived directly from the mesh,
consider quantities such as non-orthogonality of certain angles, and
mismatches between primal and dual grid components. Loosely speaking,
these measures are based on larger scale properties of the mesh, often
quantifying the extent to which neighbouring cells differ from each
other. This is equivalent to analysing higher-order spatial derivatives
in the expansion of $F$, which we do not perform in this paper.

\section{Optimally transported meshes on the sphere}
\label{sec:otsphere}

\subsection{The definition of an optimally transported mesh.}
\label{ssec:otdef}

There are various approaches for solving the Monge--Kantorovich problem
of constructing the optimal transport map that minimises the appropriate
cost function. Some approaches attack the mass transportation problem
directly, solving the problem ``from first principles''. An alternative
approach, which we describe here, reduces the problem to solving a
partial differential equation of Monge--Ampère type. This can be done
with fast algorithms which are amenable to analysis. It is shown in
\citet{mccann2001polar} that an optimal transport map on a manifold can
again be expressed in terms of the gradient of a scalar function. Let
$M$ be a general Riemannian
manifold and $u(\vec{\xi}): M \to \mathbb{R}$ be the scalar `potential'
generated by the optimal transport procedure. We then define the
\emph{McCann map} as an \emph{exponential map} $F: M \to M$:
\begin{equation}
  \label{eq:expmap}
  \vec{x} = F(\vec{\xi}) = e^{\nabla u} \vec{\xi}.
\end{equation}
The quantity $\nabla u(\vec{\xi})$ lies in the cotangent space at
$\vec{\xi}$, which can be trivially associated with the tangent space
$T_{\xi} M$. The exponential map maps this onto $M$ itself. Intuitively,
one selects the geodesic that passes through $\vec{\xi}$ and coincides
with $\nabla u$ there, then travels a distance $|\nabla u|$ along this
geodesic. The expression \cref{eq:expmap} reduces to the earlier
expression \cref{eq:xiplusgradu} if $M$ is some $\mathbb{R}^n$.

From optimal transport theory, the function $u$ automatically inherits a
convexity property: it is $c$-convex, as defined in
\citet{mccann2001polar}, where $c$ denotes the cost function used for
optimal transport. According to \citet{mccann2001polar} (see also
\citet{loeper2011regularity}), if the monitor function $m$ is
sufficiently smooth then so is the potential $u$, and so the map
\cref{eq:expmap} is well-defined and locally bijective. The McCann map
thus associates a well-defined map $F: M \to M$ with the scalar-valued
function $u(\xi)$. The map $F$ induces a well-defined area map $r(\vec{\xi})$,
which can be constructed in terms of $u$. For problems posed in
Euclidean space, as stated in \cref{eq:detIplusH}, $r(\vec{\xi})$ is given by
\begin{equation}
  r(\vec{\xi}) = \det(I + H(u)).
\end{equation}
Requiring that $m(\vec{x}) r(\vec{\xi})$ is constant over the mesh then
leads to a variant of the celebrated Monge--Ampère equation, a
fully-nonlinear second-order partial differential equation. If $M$ is
a more general manifold, the resulting expression for $r$ in terms of
$u$ involves some Monge--Ampère-like operator, conceptually
\begin{equation}
  \label{eq:chrisaug2}
  r = \operatorname{MA}(u).
\end{equation}
Setting $m(\vec{x}) r(\vec{\xi})$ to be constant then leads to an
equation of {\em Monge--Ampère-type}. In the case of the sphere $S^2$,
an explicit form of this equation is derived in \citet{mcrae2018optimal}
and is given in \cref{eq:mdetstuff}.

Conversely, suppose we use \cref{eq:UVratio} to associate an area map
$r(\xi)$ with a general map $F: M \to M$. \citet{mccann2001polar} showed
that if $r(\xi)$ is absolutely continuous with respect to the Lebesgue
measure on $M$---which is certainly true if the monitor function $m$ is
continuous and bounded away from zero---then there exists a unique
optimal transport map of the form \cref{eq:expmap} with this associated
area map.

Significantly, if $M = S^2$, the unit sphere, the geodesics are segments
of great circles, and the resulting exponential map in \cref{eq:expmap}
can be calculated easily. Indeed, it can be written in closed form as
\begin{equation}
\label{eq:rodri1}
  \vec{x} = \cos(\delta)\, \vec{\xi} + \sin(\delta)\frac{\nabla u}{|\nabla u|}, \quad \delta = |\nabla u|,
\end{equation}
or, equivalently,
\begin{equation}
\label{eq:rodri2}
  \vec{x} = \cos(\delta)\, \vec{\xi} + \frac{\sin(\delta)}{\delta} \nabla u, \quad \delta = |\nabla u|.
\end{equation}
This is a simple case of Rodrigues' rotation formula. The form
\cref{eq:rodri1} makes it clear that the destination $\vec{x}$ is an
appropriate combination of orthogonal unit vectors, while the
alternative form \cref{eq:rodri2} highlights the reduction to the planar
expression \cref{eq:xiplusgradu} in the small-$\delta$ limit. We use
these closed-form expressions for all our subsequent calculations on
$S^2$.

\subsection{A-priori estmates of the local and global regularity of optimally transported meshes}
\label{ssec:otreg}

In a series of papers (extending earlier work of, among others,
Pogorelov, Lions, Gilbarg, Trudinger, and Urbas in $\mathbb{R}^n$),
Loeper \citep{loeper2011regularity} and McCann \citep{mccann2001polar}
have derived estimates for the derivatives of the McCann map acting on
$S^2$. These can be used to help determine the mesh regularity through
\cref{eq:locskew}. The principal results from this work are summarised
as follows.
\begin{theorem}
\emph{(Trudinger et al.)} Suppose that $u$ satisfies a Monge--Ampère
equation in Euclidean space of the form
\begin{equation}
\label{eq:chrisaug3}
\det(I + H(u)) = g(\xi,\nabla u) \equiv 1/m(\xi,\nabla u).
\end{equation}
There then exists a constant $C$ that depends only on $g$, the domain
$\Omega_C$, and any boundary conditions, such that
\begin{equation}
\sup_{\Omega_C} |H(u)| \leq C.
\end{equation}
\end{theorem}
This result is then extended by Loeper:
\begin{theorem}
\emph{(Loeper)} Suppose, analogously to \cref{eq:chrisaug3}, that $u$
is the solution of a problem of Monge--Ampère-type on the sphere,
arising from an optimal transport problem as defined in
\cref{eq:chrisaug2}. As long as $g > 0$ (\emph{i.e.}, $m > 0$), then
\begin{enumerate}
  \item if $g \in C^{1,1}$ then $u \in C^{3,\alpha}$, and
  \item if $g \in C^{\infty}$ then $u \in C^{\infty}$.
\end{enumerate}
\end{theorem}

These results are significant for mesh generation: the map
from the computational domain $\Omega_C$ to the physical domain
$\Omega_P$ is given by the exponential map of the gradient of $u$. The
smoothness of $u$ therefore implies smoothness of the map. We can deduce
that if $g \in C^{\infty}$, the map is a $C^{\infty}$-function of $\xi$.
Consequently, the singular values $\sigma_j$ of the Jacobian of the map
are bounded, differentiable functions of $\xi$ over $S^2$. It follows
further from the convexity properties that both $\sigma_1$ and
$\sigma_2$ are uniformly bounded away from zero. We deduce from this
that the skewness $Q$ of the mesh is a bounded, differentiable function
of $\xi$ over $S^2$. Intuitively, we expect that the mesh we generate
cannot become too skew.

It is shown further by Loeper that the formal regularity of the function $u$ is
slightly better on the sphere than for the plane. This is because the
\emph{cost-sectional curvature}---defined in \citet{loeper2011regularity}
and closely related to the usual curvature---is uniformly positive on
the sphere, but is zero on the plane. In fact, it is possible to get
certain formal regularity results on the sphere even if the source
measure vanishes. On the sphere we also avoid problems on the plane,
seen in \citet{budd2015geometry}, where the mesh loses regularity as one
approaches the boundary. In \cref{sec:planecomp}, we make a direct
comparison between the sphere and the plane.

\subsection{A local coordinate-based approach}
\label{ssec:otcoord}

It is helpful to see how the exponential map can be expressed in terms
of a \emph{local} two-dimensional coordinate basis mapping from
$\mathbb{R}^2$ to $S^2$. A natural basis to use is spherical angles
$(\theta, \phi)$ with respect to some (unit) axis $\vec{\omega}$. This
maps a coordinate patch $(\theta, \phi) \in \mathbb{R}^2$ directly onto
$S^2$. We can then consider---locally, at least---$u \equiv u(\theta,\phi)$.
To avoid singularities in the coordinate mapping, we will assume that we
are working in a region well-separated from the poles $\pm\vec{\omega}$.

In this local basis, we have
\begin{equation}
\vec{\xi} = (\sin\theta \cos\phi, \sin\theta \sin\phi, \cos\theta)^T,
\end{equation}
with local unit vectors
\begin{equation}
\vec{e}_\theta = (\cos\theta \cos\phi, \cos\theta \sin\phi, - \sin\theta)^T, \quad
\vec{e}_\phi = (-\sin\phi, \cos\phi,0)^T;
\end{equation}
these are orthogonal and also orthogonal to $\vec{\xi}$. We then have
\begin{equation}
\label{eq:gradusph}
\nabla u = u_\theta \vec{e}_\theta + \frac{u_\phi}{\sin\theta} \vec{e}_\phi,
\end{equation}
where, by assumption, we are working on a patch where $\sin\theta$ is
bounded away from zero. Note that, using the relation
$\cos\theta = \vec{\xi}\cdot\vec{\omega}$, we can express these vectors
as
\begin{equation}
\label{eq:unitvecs}
\vec{e}_\theta = \frac{\cos(\theta) \, \vec{\xi} - \vec{\omega}}{\sin\theta}, \quad
\vec{e}_\phi = \vec{\xi} \times \vec{e}_\theta = -\frac{\vec{\xi} \times \vec{\omega}}{\sin\theta}.
\end{equation}
We can then substitute \cref{eq:gradusph} into \cref{eq:rodri2} to find
the McCann map explicitly:
\begin{equation}
\label{eq:mccannsph}
\vec{x} = \cos(\delta)\, \vec{\xi} + \frac{\sin(\delta)}{\delta} \left(u_{\theta} \vec{e}_{\theta} + \frac{u_{\phi}}{\sin\theta} \vec{e}_{\phi} \right),
\quad \delta = \sqrt{u_{\theta}^2 + \left( \frac{u_{\phi}}{\sin\theta} \right)^2 }.
\end{equation}

The map \cref{eq:mccannsph} induces a map on the local coordinate space
from $(\theta, \phi) \to (\theta', \phi')$. It follows from standard
geometry that the area ratio is
\begin{equation}
\label{eq:arearatiosph}
r(\vec{\xi}) = \frac{\sin\theta'}{\sin\theta} \; |K|, \quad \mbox{where}\ K = \frac{\partial(\theta',\phi')}{\partial(\theta,\phi)},
\end{equation}
and $|K|$ is the determinant of this. After some manipulation, it
follows from \cref{eq:mccannsph} that
\begin{equation}
\label{eq:thetaprimeexprs}
\cos\theta' = \cos\delta \cos\theta - \frac{\sin(\delta)}{\delta} \sin(\theta) \, u_\theta, \quad
\sin(\phi' - \phi) = \frac{\sin(\delta) \, u_\phi}{\delta \sin\theta' \sin\theta}
\end{equation}
The scaling factor $r$ can then be derived by differentiating
\cref{eq:thetaprimeexprs} with respect to $\theta$ and $\phi$ and
applying \cref{eq:arearatiosph}. Specifying the value of $r$ leads to a
form of the Monge--Ampère equation.

\subsection{Summary}

The formulation described in \cref{ssec:otdef} has several advantages
for mesh generation. Firstly, we need only work with scalar
quantities---monitor functions $m$ and mesh potentials $u$---in order to
compute the map $\vec{x}(\vec{\xi})$. This leads directly to flexible
and robust methods for time-dependent mesh generation which are
relatively easy to implement. Secondly, there is a substantial
body of theory giving formal regularity estimates for solutions of the
optimal transportation problem, as mentioned in \cref{ssec:otreg}. This
can be exploited to give formal regularity properties for the map and
hence the mesh. Thirdly, provided the exponential map can be calculated
easily (as on the sphere), it gives a systematic and straightforward way
of finding a map $M \to M$ which can be used to calculate the mesh in a
natural manner. This avoids the many ad-hoc approaches to mesh
construction that can be found in the literature, which often involve
fine-tuning at a local level.

\section{Exact maps generated by axisymmetric monitor functions and their regularity}
\label{sec:axisym}

For general monitor functions, the optimal transport map cannot be
expressed analytically. However, it is possible to create exact
solutions for certain classes of meshes on the sphere by considering the
maps arising from axisymmetric monitor functions on the sphere. The
action of these maps on certain computational meshes $\tau_C$ can be
studied to generate physical meshes $\tau_P$ (which need not themselves
be axisymmetric). The purpose of doing this is two-fold. Firstly, the
regularity of the resulting physical mesh $\tau_P$ can be deduced
directly from this calculation. We can then obtain exact expressions for
the scaling and skewness of the resulting meshes. The second reason for
this study is that a number of the meshes so generated are appropriate
to be used with PDE problems on the sphere, such as some of those
described in \citet{slingo2009developing}. For example, it is easy to
generate smooth meshes which can resolve specific regions of the sphere,
which may be appropriate for fronts and cyclones.  We return to the case
of calculating more general meshes in \cref{sec:numnonaxi}.

\subsection{The basic geometry of the maps}

An axisymmetric function $u$ satisfies $u(\theta,\phi) \equiv u(\theta)$,
where the coordinates are defined with respect to some axis
$\vec{\omega}$. We then have $u_\phi = 0$, so
\begin{equation}
\label{eq:graduaxi}
  \nabla u = u_\theta \vec{e}_\theta,
\end{equation}
It then follows from \cref{eq:rodri1,eq:mccannsph} that
\begin{equation}
\label{eq:mccannaxi}
\vec{x} = \cos(\delta) \, \vec{\xi} + \frac{\sin(\delta)}{\delta} u_\theta \vec{e}_\theta, \quad\mbox{where}\ \delta = |u_\theta|.
\end{equation}
It follows immediately from \cref{eq:thetaprimeexprs} that if the local
coordinates for $\vec{\xi}$ are $(\theta, \phi)$, the new local
coordinates $(\theta', \phi')$ for $\vec{x}$ are given by
\begin{equation}
\label{eq:thetaprimeaxi}
\theta' = \theta +  \mathrm{d}u/\mathrm{d}\theta, \quad \phi' = \phi.
\end{equation}
Thus, from \cref{eq:arearatiosph}, the area scaling is given by
\begin{equation}
\label{eq:areaaxi}
r(\theta, \phi) = \frac{\sin\theta'}{\sin\theta} \dd{\theta'}{\theta} = \frac{\sin\theta'}{\sin\theta} (1 + u_{\theta \theta}).
\end{equation}

It is also useful to also consider the axisymmetric map in a
coordinate-free form. As before, let $\vec{\omega}$ be the (unit) axis
of symmetry. It follows from \cref{eq:unitvecs} that
\begin{equation}
\label{eq:nablauaxi}
\nabla u = \dd{u}{\theta} \, \frac{\cos(\theta) \, \vec{\xi} - \vec{\omega}}{\sin\theta} = \left( \theta' - \theta \right) \frac{\cos(\theta) \, \vec{\xi} - \vec{\omega}}{\sin\theta}.
\end{equation}
Combining the previous results, we have
\begin{equation}
\label{eq:mccannaxialt}
\vec{x} = \cos(\theta'-\theta) \, \vec{\xi} + \sin(\theta'-\theta) \, \frac{\cos(\theta) \, \vec{\xi} - \vec{\omega}}{\sin\theta},
\end{equation}
Using \cref{eq:thetaprimeaxi,eq:mccannaxialt}, we can generate a map
$\vec{x}(\vec{\xi})$ from the sphere to itself for any suitable
$u(\theta)$ and axis of symmetry $\vec{\omega}$.

\subsection{Calculation of \texorpdfstring{$\theta'$}{theta'} from a monitor function}

Consider an axisymmetric monitor function $m(\vec{x}) \equiv m(\theta')$,
where $\cos\theta' = \vec{x}\cdot\vec{\omega}$. We can use the results
of the previous subsection to calculate the map that equidistributes this
monitor function. Using \cref{eq:areaaxi}, the equidistribution
condition \cref{eq:equi} gives
\begin{equation}
\label{eq:equiaxi}
m(\theta') \sin\theta' \dd{\theta'}{\theta} = \alpha \sin\theta
\end{equation}
where $\alpha$ is a normalisation constant. The axis of symmetry must
map to itself, implying
\begin{equation}
\label{eq:axibcs}
\theta'(0) = 0, \quad \theta'(\pi) = \pi.
\end{equation}
Integrating \cref{eq:equiaxi}, we have
\begin{equation}
\label{eq:axiintegral}
F(\theta') \equiv \int_{0}^{\theta'} m(t) \sin t \,\mathrm{d}t = \alpha (1 - \cos\theta),
\end{equation}
where the normalisation constant $\alpha$ satisfies
\begin{equation}
\label{eq:axialpha}
2 \alpha = \int_{0}^{\pi} m(t) \sin t \,\mathrm{d}t.
\end{equation}

For particular monitor functions $m(\theta')$, we can use
\cref{eq:equiaxi,eq:axibcs,eq:axiintegral,eq:axialpha} to calculate
$\theta'$ directly from $\theta$ by inverting $F$. Applying
\cref{eq:mccannaxialt} then lets us calculate $\vec{x}(\vec{\xi})$
directly. We can apply this transformation to the vertices of some
reasonably uniform mesh $\tau_C$ to generate new meshes adapted to the
given monitor function.

\subsection{The local regularity of the map}
\label{ssec:localreg}

Two important mesh properties are its scale and regularity. Here, scale
refers to the sizes of its cells, per \cref{eq:locscale}, and regularity
to the skewness $Q$ of cells, per \cref{eq:locskew}. Both of these are
local measures of mesh quality, and relate to the resulting errors which
can be expected when solving PDEs on the mesh. For these axisymmetric
maps, we can obtain analytic expressions for the scale and skewness
quantities in terms of the monitor function.

Consider a local quadrilateral coordinate patch in $\mathbb{R}^2$
centred on $(\theta, \phi)$ and of sides $\delta \theta \times \delta \phi$.
This corresponds to a patch in $S^2$ of sides
$\delta x \times \delta y = \sin\theta\ \delta \phi \times \delta \theta$.
The patch in $\mathbb{R}^2$ is mapped to a patch centred on
$(\theta', \phi')$ of sides $\delta \theta' \times \delta \phi'$, which
corresponds to a patch on $S^2$ of sides
$\delta x' \times \delta y' = \sin\theta'\ \delta \phi' \times \delta \theta'$.
In the limit of the patch going to zero,
\begin{equation}
\dd{x'}{x} = \frac{\sin\theta'}{\sin\theta} \dd{\phi'}{\phi}, \quad
\dd{y'}{y} = \dd{\theta'}{\theta}.
\end{equation}
For an axisymmetric equidistribution map generated by the monitor
function $m(\theta')$, \cref{eq:equiaxi} implies
\begin{equation}
\dd{\phi'}{\phi} = 1, \quad
\dd{\theta'}{\theta} = \frac{\alpha}{m(\theta')} \frac{\sin\theta}{\sin\theta'}.
\end{equation}
Combining these, we have
\begin{equation}
\dd{x'}{x} = \frac{\sin\theta'}{\sin\theta}, \quad
\dd{y'}{y} = \frac{\alpha}{m(\theta')} \frac{\sin\theta}{\sin\theta'}.
\end{equation}
It follows that the Jacobian matrix, $J$, of the map from the local
patches of $S^2$ is diagonal and has eigenvalues or singular values
\begin{equation}
\label{eq:eigeqn}
\sigma_1 = \frac{\sin\theta'}{\sin\theta}, \quad \sigma_2 = \frac{\alpha}{m(\theta')} \frac{\sin\theta}{\sin\theta'}.
\end{equation}
From this, we can deduce expressions for the scaling and skewness of the
resulting map.

\begin{lemma}
\label{thm:axireg}
\emph{(i)} The local scaling of the map is given by
\begin{equation}
\label{eq:axiscale}
s = \sigma_1 \sigma_2 = \alpha/m.
\end{equation}
\emph{(ii)} The local skewness of the map is given by
\begin{equation}
\label{eq:Qaxi}
Q = \frac{1}{2} \left(\frac{\alpha}{m(\theta')} \frac{\sin^2\theta}{\sin^2\theta'} + \frac{m(\theta')}{\alpha}\frac{\sin^2\theta'}{\sin^2\theta} \right).
\end{equation}
\end{lemma}
\begin{proof}
Result \emph{(i)} follows directly from \cref{eq:eigeqn}, and is
consistent with the equidistribution condition \cref{eq:equi}.
Similarly, \emph{(ii)} follows from \cref{eq:eigeqn} and \cref{eq:locskew}.
\end{proof}

We now show that $Q \to 1$ at the poles, as long as $m$ is continuous.
This implies that $Q$ is close to unity for open regions around each
pole. Therefore, the resulting mesh is extremely regular in these
regions even if there is significant mesh contraction.

\begin{lemma}
\label{thm:poles}
If $m$ is continuous at the poles, $Q \to 1$ as $\theta' \to 0$ or
$\theta' \to \pi$.
\end{lemma}
\begin{proof}
It follows from the regularity of the map that
$\mathrm{d}\theta'/\mathrm{d} \theta$ exists. Furthermore, since
$\theta' = 0$ when $\theta = 0$, it follows that
$\sin\theta'/\sin\theta \to \mathrm{d}\theta'/\mathrm{d}\theta$ as
$\theta \to 0$. Now, as in \cref{eq:equiaxi}, the equidistribution
condition leads to
\begin{equation}
m \sin\theta' \dd{\theta'}{\theta} = \alpha \sin\theta.
\end{equation}
It follows that, as $\theta \to 0$, we have
\begin{equation}
\frac{m}{\alpha} \frac{\sin^2\theta'}{\sin^2\theta} \to \frac{m}{\alpha} \frac{\sin\theta'}{\sin\theta} \dd{\theta'}{\theta} = 1.
\end{equation}
Thus $Q \to 1$. A similar calculation can be performed for
$\theta \to \pi$.
\end{proof}
We immediately see, stemming from the fact that $S^2$ is compact and the
result of \cref{thm:poles}, that controlling the variation of $Q$ is
easier than on the plane. The challenge of designing a mesh adaptation
strategy in the axisymmetric case can thus (at least locally) be
summarised as an ODE-constrained optimisation problem.

\textbf{Challenge 3} Determine a suitable monitor function $m(\theta')$,
satisfying \cref{eq:equiaxi}, so that with scaling $s$ and mesh skewness
$Q$ given by \cref{eq:axiscale} and \cref{eq:Qaxi}, we can obtain bounds
for the solution error.

A general answer to  Challenge 3 is difficult and we leave it as a
subject for future research.

\section{Examples of meshes generated by using axisymmetric monitor functions}
\label{sec:numaxi}

\subsection{Overview of the meshes generated}
In this section, we consider two examples of monitor functions which are
axisymmetric about an axis $\omega$, and the maps, and hence meshes,
that these induce. These monitor functions are chosen to give meshes
which can be analysed and are of potential practical importance for
meteorological applications. The examples will be (i) meshes which
concentrate points in localised regions, and (ii) meshes which
concentrate points in rings. (We note that very similar methods, with a
monitor function such as $m(\theta') = \gamma e^{-\gamma(1 - \cos\theta')} + 1/2$,
where $\gamma$ is assumed to be large, can be used to concentrate mesh
points close to a single point on $S^2$.)

For any axisymmetric monitor function $m(\theta')$, the formula
\cref{eq:axiintegral} can be used to evaluate the map $\theta'(\theta)$.
The desired meshes $\tau_P$ are then produced by applying the resulting
axisymmetric maps to various computational meshes $\tau_C$, such as
cubed-sphere and icosahedral meshes, with known topology and
connectivity. The definition of the skewness, as in
\cref{eq:locskew}, is a property of the \emph{map}, not of the meshes
themselves. However, as long as $\tau_C$ is reasonably uniform and has
regular angles, such as the examples considered above, large skewness
values will coincide with highly-skew cells in the adapted mesh
$\tau_P$, and low skewness with regular cells. We can hence estimate the
regularity of the resulting mesh.

\subsection{Meshes concentrating points into regions}

\subsubsection{Analytical construction of the meshes}

We firstly consider monitor functions which induce meshes that
concentrate more points uniformly into specified regions, such as a disc
centred on the axis $\vec{\omega}$. In a meteorological context, this
mesh could be used to represent a localised feature such as a moving
hurricane or a vortex patch, or a static feature such as a country
\citep{ringler2011exploring}. Ideally the algorithm for constructing such
a mesh will resolve the region in fine detail without making the mesh
in other regions too coarse or irregular, or introducing too much
skewness in the transition between the two regions. Using optimal
transport methods, we can produce and analyse such meshes. To achieve
this, we consider both a discontinuous ``top-hat'' monitor function, for
which we can express the map analytically, and a smoothed version of
this.

The \emph{top-hat} monitor function is given by the expression
\begin{equation}
\label{eq:mtophat}
m(\theta') =
\begin{cases}
\rho_1, &\theta' < \Theta'\\
\rho_2, &\theta' > \Theta',
\end{cases}
\end{equation}
where $\Theta'$ marks the boundary between high- and low-resolution
regions. The ratio $\gamma = \rho_2/\rho_1$ sets the ratio of mesh
density between the two regions. We will assume that $\gamma < 1$,
indeed for meshes with high compression, we expect that $\gamma$ will be
small. The top-hat function is discontinuous, so formally we have less
regularity than discussed in \cref{ssec:otreg}, but the resulting
optimal transport map is still continuous. Integrating \cref{eq:equiaxi}
from the `boundaries' 0 and $\pi$, we have
\begin{align}
\label{eq:tophatreg1}
\rho_1 (1 - \cos\theta') = \alpha (1 - \cos\theta), \quad \theta' < \Theta', \\
\label{eq:tophatreg2}
\quad \rho_2 (1 + \cos\theta') = \alpha (1 + \cos\theta), \quad \theta' > \Theta'.
\end{align}
It follows immediately from \cref{eq:axialpha} that the normalisation
constant $\alpha$ satisfies
\begin{equation}
2 \alpha = \rho_1 (1 - \cos\Theta') + \rho_2 (1 + \cos\Theta').
\end{equation}
Finally, define $\Theta$ to be the preimage of $\Theta'$ under the map.
By invoking continuity of the mesh, it follows that
\begin{equation}
\frac{\rho_1 (1 - \cos\Theta')}{1 - \cos\Theta} = \frac{\rho_2 (1 + \cos\Theta')}{1 + \cos\Theta}.
\end{equation}
Using standard trigonometrical identities, this can be written
\begin{equation}
\label{eq:tophatbdry}
\rho_1 \tan^2(\Theta'/2) = \rho_2 \tan^2(\Theta/2).
\end{equation}

\begin{figure}[!tb]
  \centering
  \includegraphics[width=0.6\columnwidth]{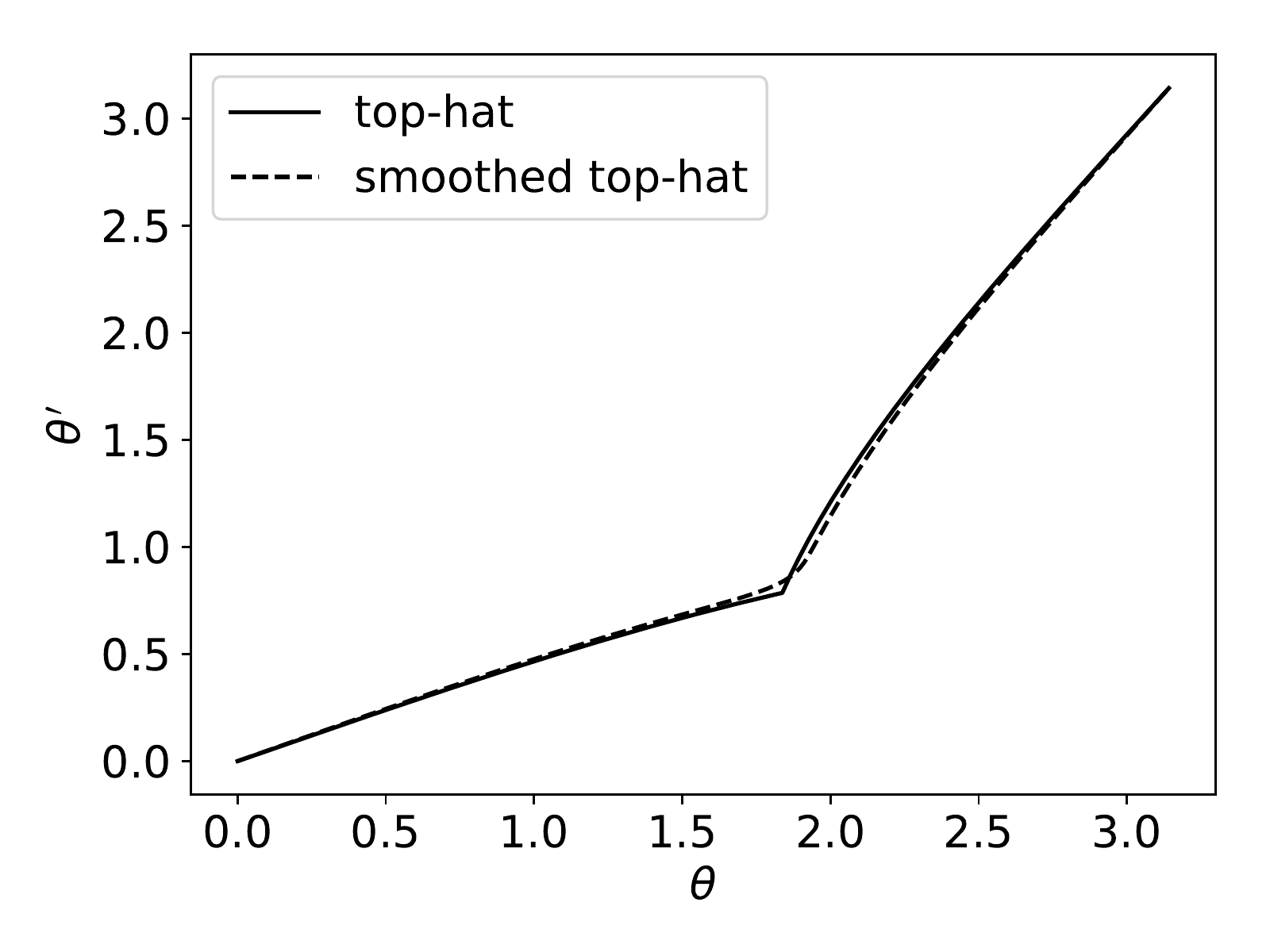}
  \caption{The map $\theta'(\theta)$ produced by the top-hat monitor
function \cref{eq:mtophat}, with $\rho_1/\rho_2 = 10$ and
${\Theta' = \pi/4}$, and by a smoothed approximation \cref{eq:mringler},
with ${\epsilon = \pi/50}$. The top-hat monitor function has a
discontinuity in $m$, which leads to a visible discontinuity in
$\dd{\theta'}{\theta}$, per \cref{eq:equiaxi}. In the smoothed top-hat,
the transition occurs over a distance $\mathcal{O}(\epsilon)$.}
\label{fig:m_tophat}
\end{figure}

Given a monitor function of the form \cref{eq:mtophat}, we can find
$\Theta$ from $\Theta'$ via \cref{eq:tophatbdry}. We then apply
\cref{eq:tophatreg1,eq:tophatreg2} to find $\theta'$ as a function of
$\theta$ in the separate ranges $\theta < \Theta$ and $\theta > \Theta$.
Having determined $\theta'$, we can calculate the image point $\vec{x}$
for a given $\vec{\xi}$ by using \cref{eq:mccannaxialt}.
For example, if $\rho_1 = 10, \rho_2 = 1$ and $\Theta' = \pi/4$ then
\begin{equation}
\alpha = 2.318\ldots , \quad \Theta = 1.837\ldots
\end{equation}
The resulting map $\theta \to \theta'$ is given in \cref{fig:m_tophat}.

A smoothed form of the top-hat monitor function was proposed by
\citet{ringler2011exploring} in the context of a global weather
forecasting model with increased resolution over the United States. In
that paper, a mesh generation algorithm using Lloyd's algorithm was
considered. The same monitor function was also used as a test problem in
the recent paper \citet{weller2016mesh} \footnote{The expression in the
original paper is incorrect; a correct version is given in our recent
paper \citet{mcrae2018optimal}, and is used here.}. This smoother
monitor function takes the form
\begin{equation}
\label{eq:mringler}
m(\theta') = \sqrt{\frac{1 - \gamma^2}{2} \left( \tanh \left( \frac{\Theta' - \theta'}{\epsilon} \right) + 1 \right) + \gamma^2},
\end{equation}
where we assume $\epsilon$ and $\gamma$ are small. For the region
$\theta' < \Theta'$ we have $m(\theta') \approx 1 = \rho_1$, while for
$\theta' > \Theta'$, $m(\theta') \approx \gamma = \rho_2$. This monitor
function therefore has a similar profile to the top-hat monitor function
\cref{eq:mtophat}, but with a smooth transition over a lengthscale
$\epsilon$. There is no closed-form integral of $m(\theta')\sin(\theta')$,
so we cannot use \cref{eq:axiintegral} to get a closed-form expression
for $\theta'(\theta)$. However, we can use numerical quadrature to
obtain an arbitrarily good approximation. We expect that it will have
similar behaviour to the top-hat monitor function if the parameters are
chosen carefully. For comparison, we present a calculation for the
smoothed monitor function, taking $\gamma = 1/10$, $\Theta' = \pi/4$,
and $\epsilon = \pi/50$, and the resulting map $\theta \to \theta'$ is
given in \cref{fig:m_tophat}.

We observe that the general piecewise-constant monitor function
\begin{equation}
\label{eq:piececonst}
m(\theta') = \rho_i, \quad \theta_i < \theta' < \theta_{i+1}, \quad i = 1, \ldots, N,
\end{equation}
or smoother versions of this, can be used to concentrate points in
annular regions, such as close to the equator or in the tropical zones.
(Examples of meshes which concentrate points in equatorial regions
are given in \citet{iga2017equatorially}).
The calculations (and indeed the resulting mesh regularity) for the
monitor function \cref{eq:piececonst} are very similar to those for the
top-hat monitor function \cref{eq:mtophat}.

\subsubsection{Analytical estimates of the regularity of the regional meshes}

We now study the regularity of the resulting maps by calculating the
skewness function $Q$. We first consider the map induced by the top-hat
monitor function \cref{eq:mtophat}. It follows from \cref{eq:Qaxi} that
the local skewness of the map is given by
\begin{equation}
Q = \frac{1}{2} \left(\frac{\alpha}{\rho_i} \frac{\sin^2\theta}{\sin^2\theta'} + \frac{\rho_i}{\alpha}\frac{\sin^2\theta'}{\sin^2\theta} \right)
\end{equation}
in each region $i=1,2$. In \cref{fig:Q_tophat}, we plot $Q$ as a
function of $\theta'$ for the case studied previously, where we have
large mesh compression with $\rho_1/\rho_2 = 10$, and $\Theta' = \pi/4$.
We see that $Q$ takes its largest value just
outside the ``top-hat region", \emph{i.e.}, at $\theta' = \Theta'_+$. We
therefore expect to see the most significant mesh distortion in this
region, as was also observed in \citet{weller2016mesh}. The value of
$Q_\mathrm{max} \approx 2.273$, implies the resulting mesh will have
some \emph{moderately} skew cells. However, this is in the context of
significant mesh compression and thus greatly enhanced resolution of the
underlying solution in the inner region.

When we consider the smoother top-hat monitor function
\cref{eq:mringler}, with the same parameters as before, the resulting
skewness factor is smaller and is plotted in \cref{fig:Q_tophat}. The
skewness $Q$ now takes its maximum value $Q_{max} \approx 1.6$ in the
outer part of the transition region. We observe that $Q = 1$ in the
middle of the transition region, implying that the mesh is very regular
there. This can be explained by continuity: approaching the transition
from inside, the cells are stretched in one direction (zonally), but are
stretched in the other direction (meridionally) when approaching from
outside. By continuity, there must be some intermediate value of
$\theta'$ where the cells are stretched equally in both directions, so
that $Q = 1$.

\begin{figure}[!tb]
  \centering
  \includegraphics[width=0.6\columnwidth]{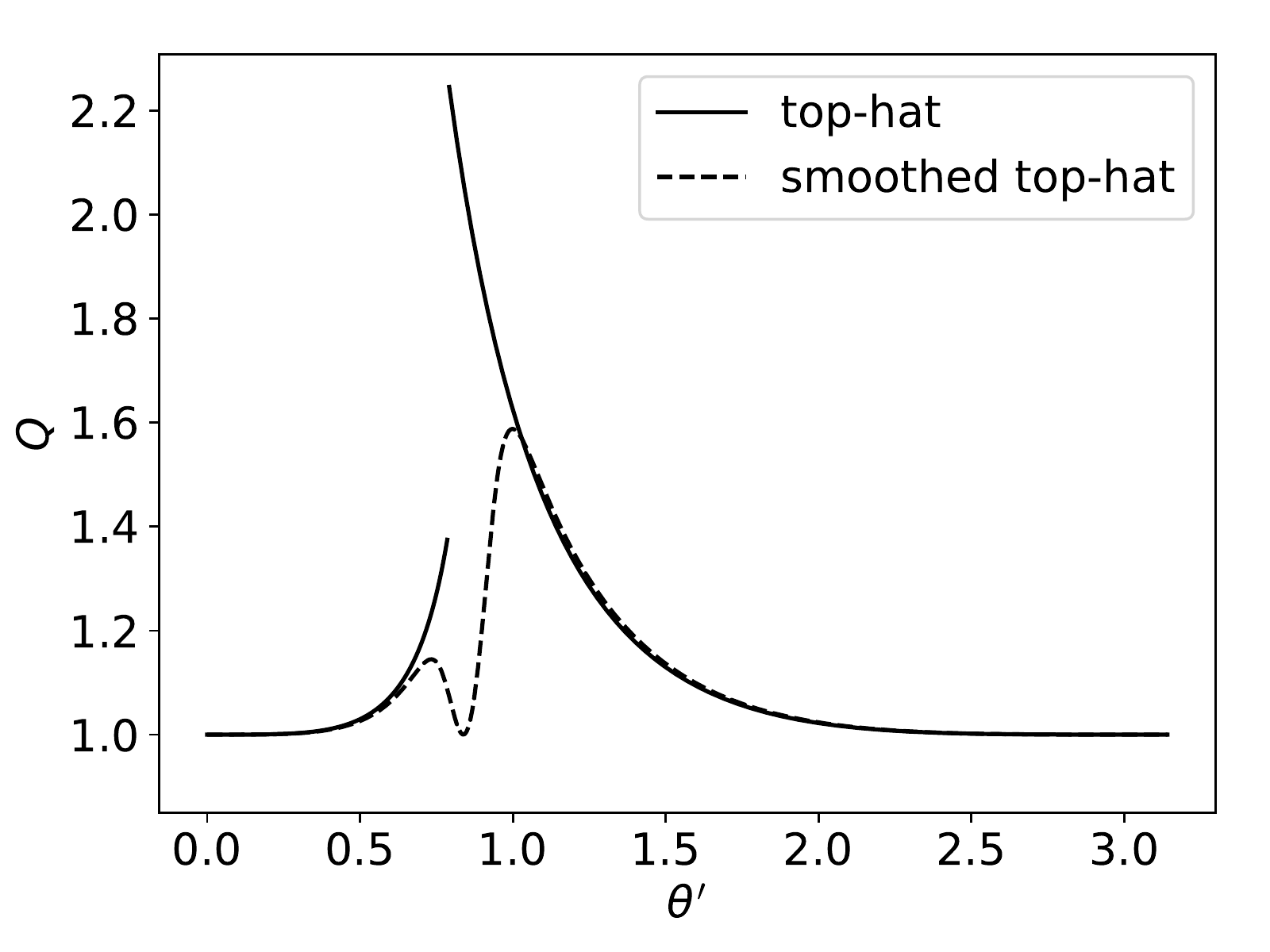}
  \caption{The skewness functions for the maps generated by the top-hat
and smoothed top-hat monitor functions. For the top-hat monitor
function, $Q$ is maximised when approaching the transition from the
outside; it can be shown analytically that the resulting
$Q_\mathrm{max} \approx 2.273$. For the smoothed approximation,
$Q_\mathrm{max}$ is notably smaller, despite the maps in
\cref{fig:m_tophat} being very similar. Perhaps surprisingly, $Q$
returns to 1 in the interior of the transition region; this can be
attributed to continuity. For both monitor functions, as expected,
$Q \to 1$ at the poles.}
\label{fig:Q_tophat}
\end{figure}

\subsubsection{The resulting regional meshes}

In this and in the next subsection, the adapted meshes are generated as
follows. For some vertex of $\tau_C$, located at $\vec{\xi} \in S^2$, we
calculate the corresponding value of $\theta$ from the expression
$\cos\theta = \vec{\xi}\cdot\vec{\omega}$. Applying the axisymmetric map gives
the value of $\theta'$ for the image point $\vec{x}$. We can then use
the expression \cref{eq:mccannaxialt} to determine $\vec{x}$ explicitly
-- it is on the same great circle as $\vec{\omega}$ and $\vec{\xi}$, and
is at an angle $\theta'$ from $\vec{\omega}$ (on the same ``side'' as
$\vec{\xi}$). This is performed for each vertex of $\tau_C$; the image
vertices form the adapted mesh $\tau_P$, which has the same connectivity
as $\tau_C$.

We use this method to look at the action of the resulting maps on three
types of meshes: a cubed-sphere mesh, an icosahedral mesh, and a
latitude--longitude mesh. We also calculate the resulting skewness for
the top-hat and smoothed top-hat monitor functions. For ease of
visualisation, the axis of rotational symmetry, $\omega$, is
always taken to be proportional to ${(0.7, -1.0, 2.0)^T}$, and the
meshes are viewed from the negative-$y$ direction. This allows the mesh
behaviour around the pole $\omega$, and in any ``inner region", to be
seen clearly. The mesh behaviour around the opposite pole, $-\omega$,
cannot be seen directly, but (in all cases that we use) it is visible
that the mesh is becoming increasingly regular towards the pole. This is
consistent with \cref{thm:poles}. The computational meshes each have
some rotational symmetry about the $z$-axis ${(0.0, 0.0, 1.0)^T}$.
However, they have no symmetry about $\omega$, the symmetry axis of the
monitor function. This is done deliberately, in order to illustrate the
more general behaviour.

\Cref{fig:th-cs-meshes} shows the effect of the resulting maps on a
cubed-sphere mesh, made up of six patches of `squares', while
\cref{fig:th-ic-meshes} shows the same for an icosahedral mesh, formed
of triangles. In each case, the high concentration of mesh points, and
resulting mesh compression, in the inner region is clear. We also see,
as predicted from \cref{thm:poles}, that there is good regularity of the
meshes near both poles, which follows from $Q$ approaching unity there.
The meshes produced by the top-hat monitor function have a visible sharp
transition at $\theta' = \Theta'$; however this transition is rather
smoother for the tanh-based monitor function. As predicted, there is a
narrow band of `regular' cells in the transition region, where
$Q \approx 1$, and the skewness of the mesh is modest elsewhere. In
general, these meshes are suitable for a computation; if desired, the
transition could be further smoothed by increasing $\epsilon$.
We also consider the map applied to a latitude--longitude mesh, which
already has a large variation in cell-size. The resulting mesh is shown
in \cref{fig:th-ll-meshes}. This no longer has orthogonal grid lines,
one of the main advantages of the latitude--longitude mesh, while the
disadvantage of large variations in cell-size is still present. This
mesh is unlikely to be useful for numerical weather prediction
calculations, and the cubed-sphere or icosahedral meshes are probably
far more suitable.

\begin{figure}[!tb]
  \centering
  \includegraphics[width=0.45\columnwidth]{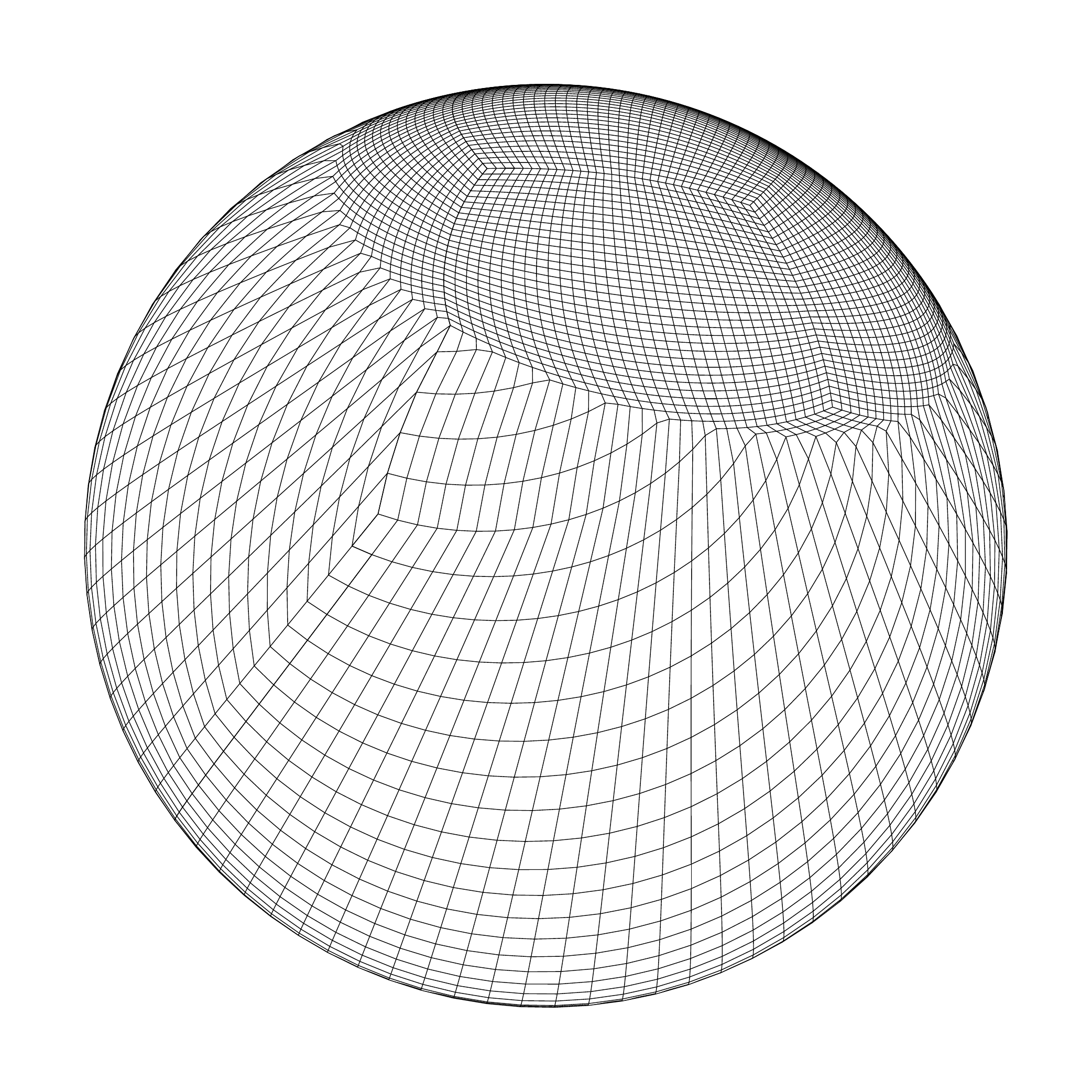}
  \includegraphics[width=0.45\columnwidth]{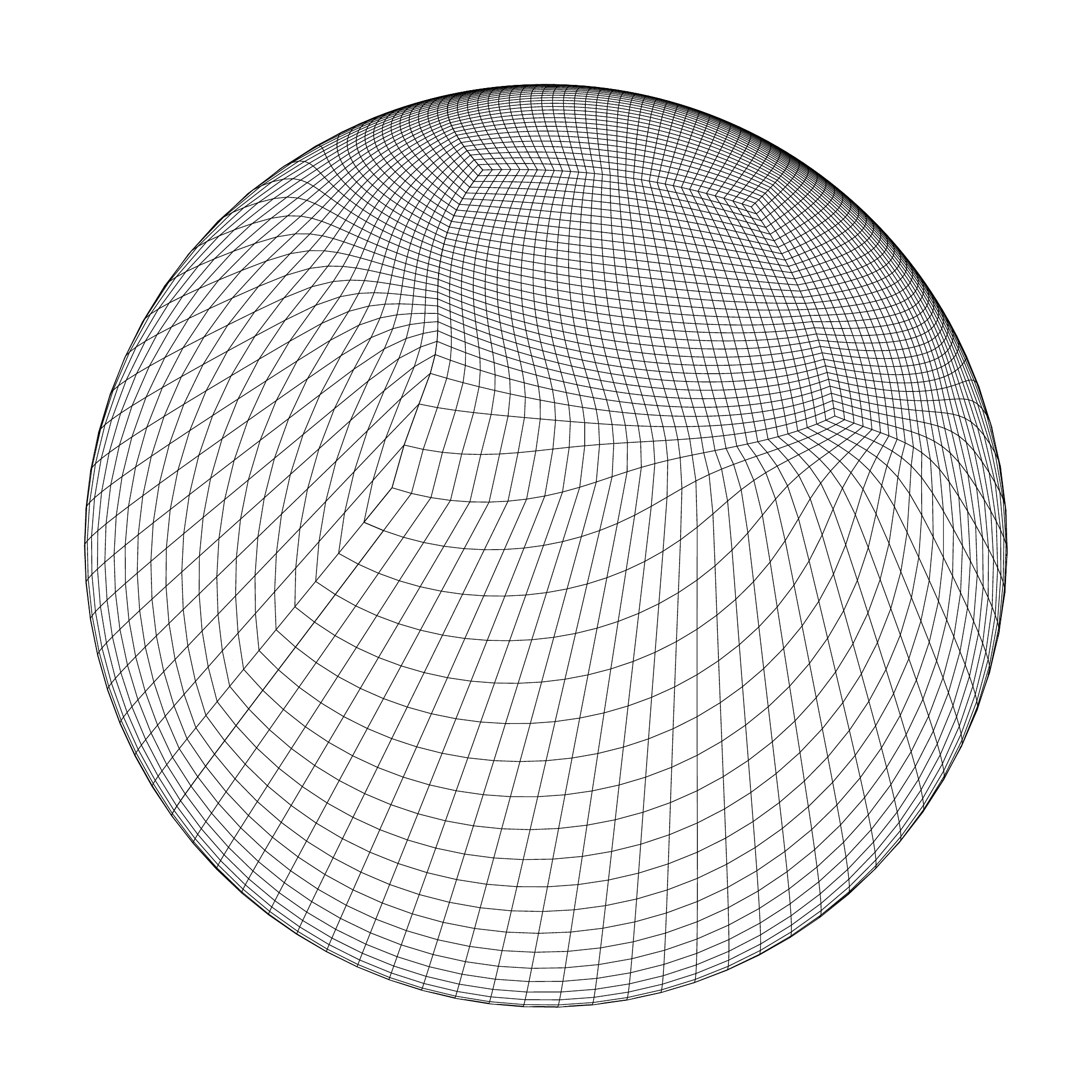}
  \caption{A cubed-sphere mesh, adapted to the top-hat and smoothed
top-hat monitor functions.}
\label{fig:th-cs-meshes}
\end{figure}

\begin{figure}[!tb]
  \centering
  \includegraphics[width=0.45\columnwidth]{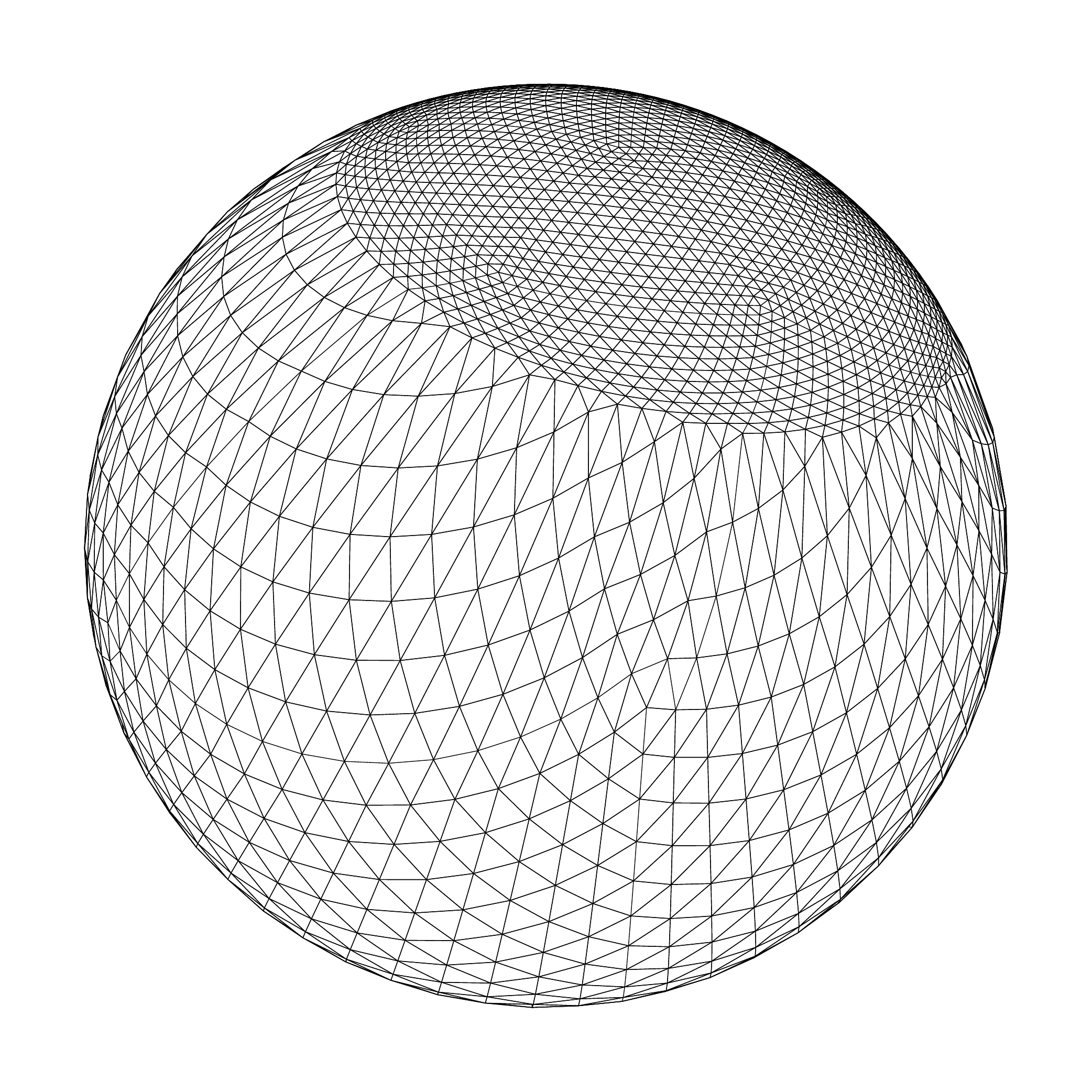}
  \includegraphics[width=0.45\columnwidth]{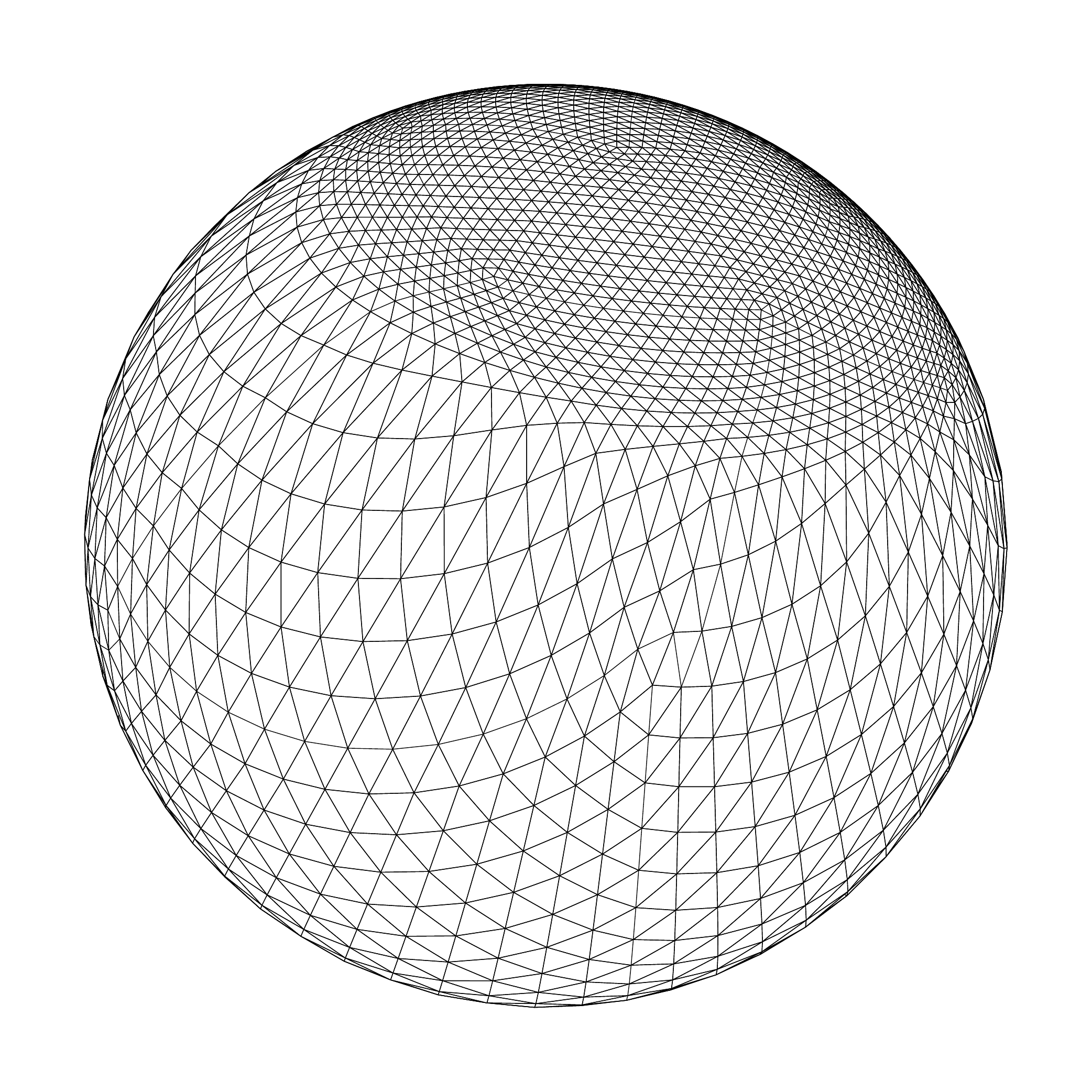}
  \caption{An icosahedral mesh, adapted to the top-hat and smoothed
top-hat monitor functions.}
\label{fig:th-ic-meshes}
\end{figure}

\begin{figure}[!tb]
  \centering
  \includegraphics[width=0.45\columnwidth]{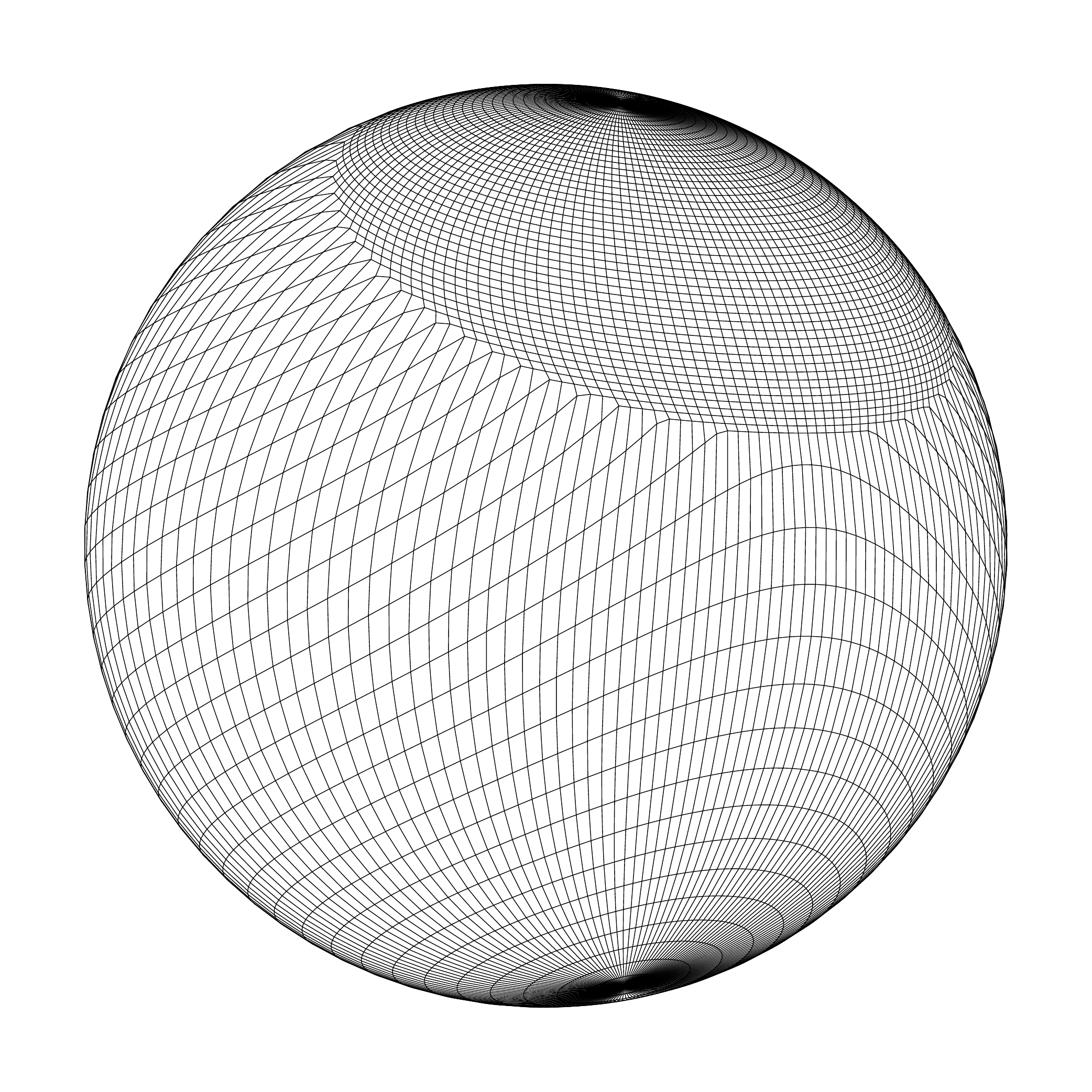}
  \includegraphics[width=0.45\columnwidth]{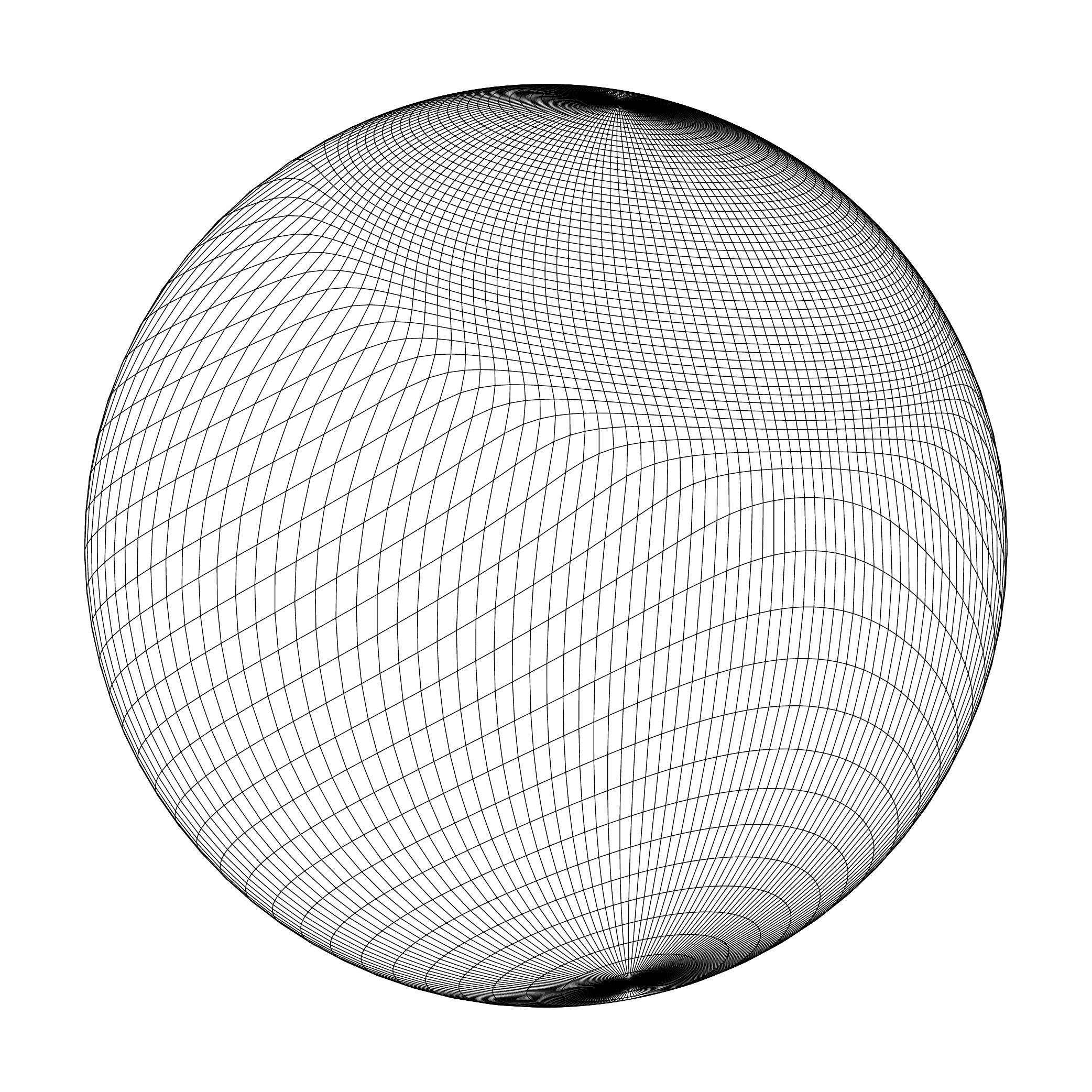}
  \caption{A latitude--longitude mesh, adapted to the top-hat and
smoothed top-hat monitor functions.}
\label{fig:th-ll-meshes}
\end{figure}

\subsection{Meshes concentrating points in rings}

\subsubsection{Analytical construction of the meshes}

Ring-like structures, or, more generally, solutions of PDEs with
features strongly aligned in a certain direction, arise naturally in
many applications. These include certain types of laser-driven optical
phenomena \citep{fibich2015nonlinear}, cross sections of bubbles, and
also in low-amplitude Rossby waves of the form seen in
\citet{slingo2009developing}. A ring can also be regarded as an extended
form of a locally anisotropic and strongly aligned feature such as a
weather front. Understanding the skewness and alignment properties of an
analytically-generated mesh for a ring gives us insight into the
corresponding features of a numerically generated mesh used to follow
such a front, see for example \citet{budd2013monge}. It is known that,
in such cases, it is vital for representing the solution accurately that
the mesh is aligned with the anisotropic feature, even if this comes at
the expense of a certain amount of mesh skewness
\citep{huang2011adaptive}.

We can consider such a feature to be one in which information is
concentrated in a neighbourhood of width $\mathcal{O}(\epsilon) \ll 1$
of a circle at an angle $\Theta'$ from the axis $\vec{\omega}$. A
monitor function leading to meshes which concentrate points in such a
ring is given by
\begin{equation}
\label{eq:mring}
m(\theta') = 1 + \frac{\beta}{\epsilon} \sech^2\left(\frac{\theta'^2 - \Theta'^2}{\epsilon}\right).
\end{equation}
The parameter $\beta$ in this expression controls the density of mesh
points in the ring. In the limit of $\epsilon \to 0$, $m$ approximates
an expression involving a delta function given by
\begin{equation}
\label{eq:mdelta}
m(\theta') = 1 + \lambda \; \delta(\theta' - \Theta').
\end{equation}
Here, $\lambda$ controls the density of mesh points in the ring, and
comparison with \cref{eq:mring}, ${\lambda = \beta /\Theta'}$ to leading
order. The analysis of the monitor function \cref{eq:mdelta} is
more straightforward than that of \cref{eq:mring} and gives a
leading-order expression for the mesh generated by \cref{eq:mring},
although the resulting mesh transformation for \cref{eq:mdelta} (or,
rather, its inverse) is discontinuous. By \cref{eq:axialpha}, we have
\begin{equation}
\label{eq:deltaalpha}
\alpha = 1 + \frac{\lambda}{2} \sin\Theta'.
\end{equation}

Integrating from 0 and from $\pi$, we then have
\begin{align}
\label{eq:deltareg1}
1 - \cos\theta' = \alpha (1 - \cos\theta), \quad \theta' < \Theta', \\
\label{eq:deltareg2}
1 + \cos\theta' = \alpha (1 + \cos\theta), \quad \theta' > \Theta'.
\end{align}
The monitor function \cref{eq:mdelta} leads to a jump in the value of the
computational coordinate $\theta$ when the physical coordinate satisfies
$\theta' = \Theta'$. If $\theta_1$ and $\theta_2$ are mapped to the
inner and outer edge of the ring, so that $\theta_1 \to \Theta'_-$ and
$\theta_2 \to \Theta'_+$, then \cref{eq:deltareg1,eq:deltareg2} give
\begin{equation}
\cos\theta_1 = 1 - \frac{1 - \cos(\Theta')}{\alpha}, \quad
\cos\theta_2 = \frac{1 + \cos(\Theta')}{\alpha} - 1
\end{equation}
For example, if we take $\Theta' = \pi/4$ and $\gamma = 5$, then
$\theta_1 = 0.464\ldots$, $\theta_2 = 1.964\ldots$, $\alpha = 2.768\ldots$,
and the resulting map is presented in \cref{fig:m_ring}.

Returning to the smoother monitor function \cref{eq:mring}, we take
$\Theta' = \pi/4$, $\beta = 5\pi/4$ -- compatible with the value of
$\lambda$ used in the delta function example -- and $\epsilon = \pi/50$.
The relevant integrals are evaluated with numerical quadrature, and the
resulting map is shown in \cref{fig:m_ring}.

\begin{figure}[!tb]
  \centering
  \includegraphics[width=0.6\columnwidth]{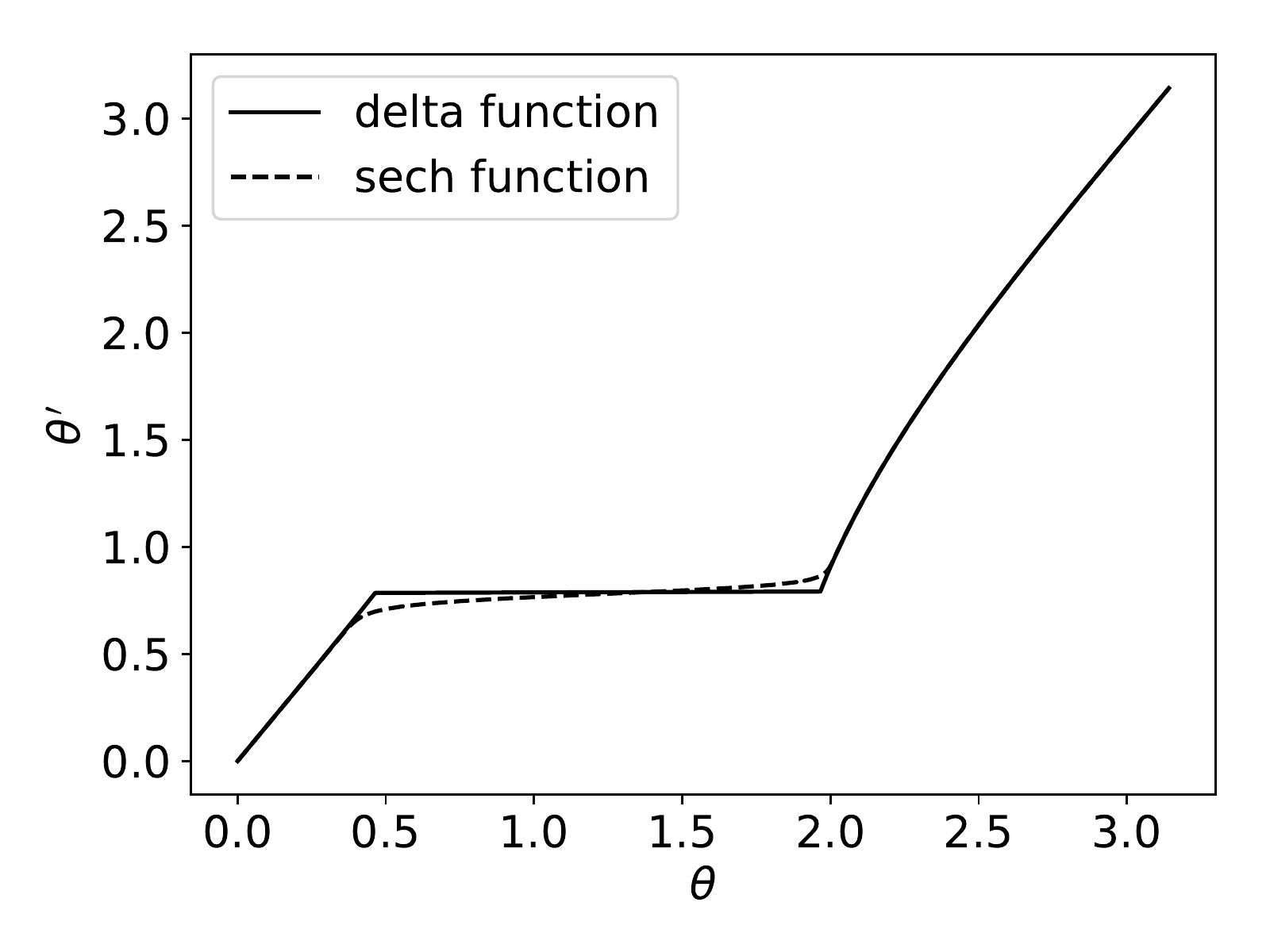}
  \caption{The map $\theta'(\theta)$ produced by a ring monitor function
based on a delta function \cref{eq:mdelta}, with ${\lambda = 5}$ and
${\Theta' = \pi/4}$, and one based on a smoother sech function
\cref{eq:mring}, with ${\epsilon = \pi/50}$. For the
delta-function-based monitor function, all
$\theta \in [\theta_1, \theta_2]$ are mapped to $\Theta'$. For the
smoother approximation, a similar transition occurs, but now over a
distance (in $\theta'$) of $\mathcal{O}(\epsilon)$.}
\label{fig:m_ring}
\end{figure}

\subsubsection{Analytical estimates of the regularity of the ring meshes}

\begin{figure}[!tb]
  \centering
  \includegraphics[width=0.6\columnwidth]{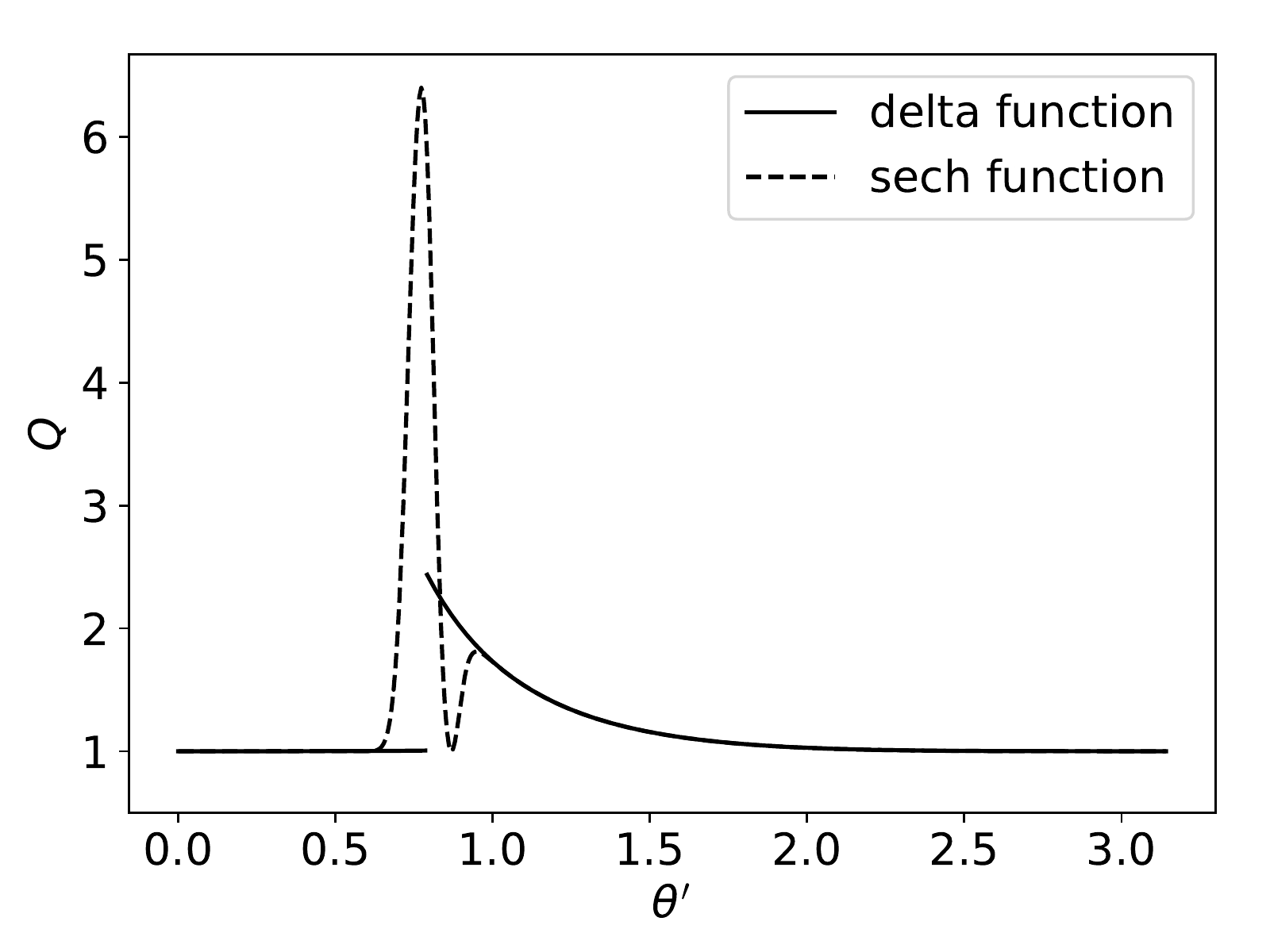}
  \caption{The skewness functions for the maps generated by the
delta-function-based and sech-function-based monitor functions for
ring-like features. For the delta-function example, we have not tried to
represent the infinite skewness in the ring itself. The maximum skewness
outside the ring is approximately 2.467. Intuitively, the ring
`swallows' cells due to the delta function. The remaining cells must
then be stretched out to cover the entire area, leading to skewness. For
the smooth example, the maximum skewness in the ring region is around
6.4. There is a secondary peak in $Q$ outside the ring, as in the
delta-function case, and $Q$ briefly touches 1 due to continuity.}
\label{fig:Q_ring}
\end{figure}

We firstly consider the smoother monitor function \cref{eq:mring}
with the same parameters as before. The skewness factor $Q$ is plotted
in \cref{fig:Q_ring}. Again, $Q$ is continuous, taking the value of 1 at
the poles. Unsurprisingly it takes its maximum value inside the ring
region. This value, close to 6.4 (representing anisotropic stretching by
a factor of between 12 and 13), is fairly moderate, given that the
monitor function varies by a factor of 63.5. We note though, that due to
continuity, $Q$ briefly touches 1 close to the ring. The resulting mesh
cells are, as we will see later,
well-aligned with the ring itself, and such a mesh is well-suited for
representing a function that is aligned with the ring
\citep{huang2011adaptive,budd2018inprep}. This is important: when
calculating such anisotropic functions, the alignment of the mesh with
the features of the function can lead to interpolation error estimates
which are much lower than those arising from a uniform mesh
\citep{huang2011adaptive}. In a sense, the good features of the
alignment outweigh the bad features of the consequent mesh skewness.

We also consider the delta function \cref{eq:mdelta}, with $Q$ plotted
in \cref{fig:Q_ring}. This function is theoretically infinite in the
ring, as cells have length zero in the meridional direction. However,
the maximum skewness outside of this is just 2.467..., obtained by
substituting $m(\theta') = 1$, $\theta' = \Theta'$, $\theta = \theta_2$,
and $\alpha$ in \cref{eq:Qaxi}. The values of $\theta_2$ and $\alpha$
are the same as those used previously.

\subsubsection{The resulting ring meshes}

\begin{figure}[!tb]
  \centering
  \includegraphics[width=0.45\columnwidth]{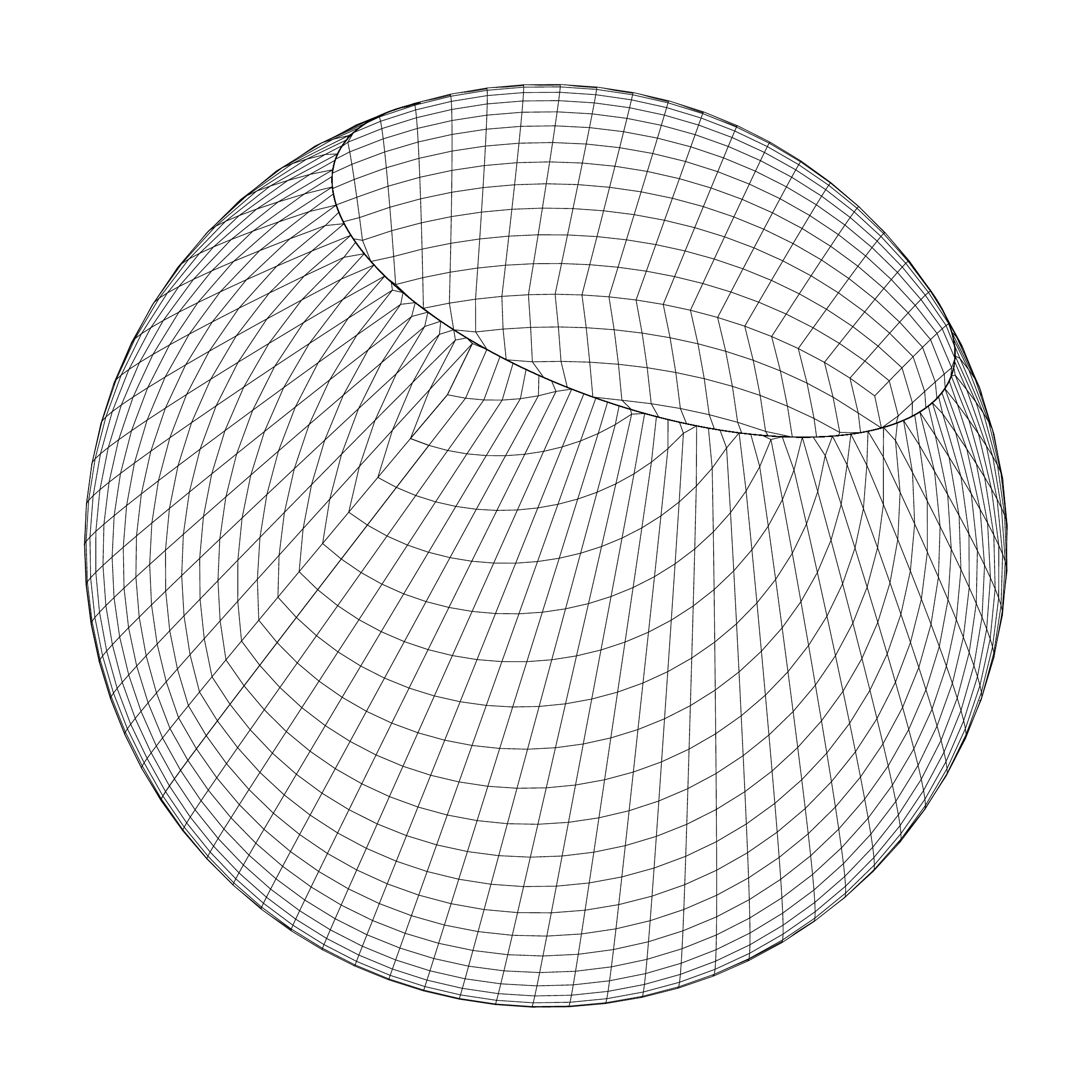}
  \includegraphics[width=0.45\columnwidth]{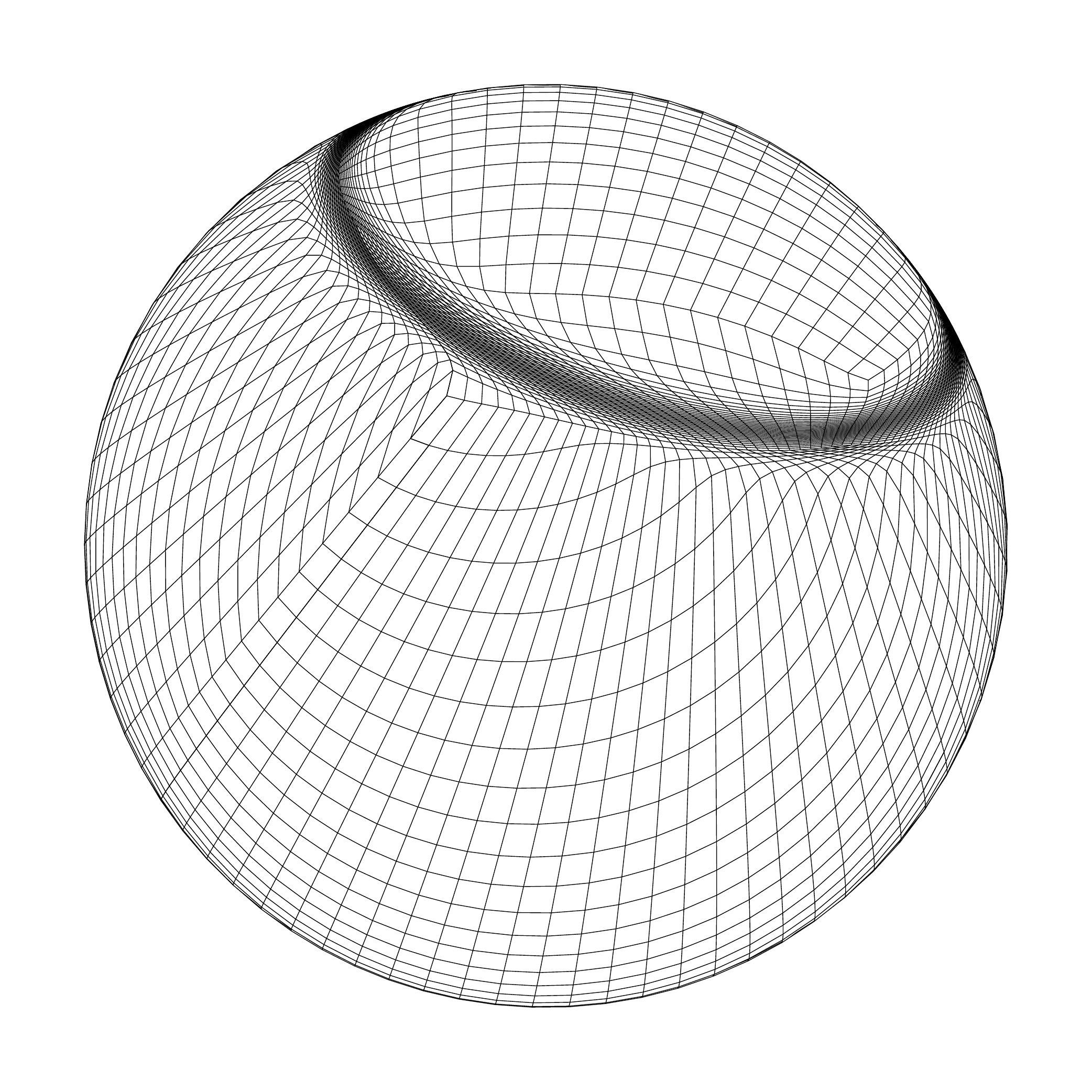}
  \caption{A cubed-sphere mesh, adapted to the delta-function-based and
sech-function-based ring monitor functions.}
\label{fig:ri-cs-meshes}
\end{figure}

\begin{figure}[!tb]
  \centering
  \includegraphics[width=0.45\columnwidth]{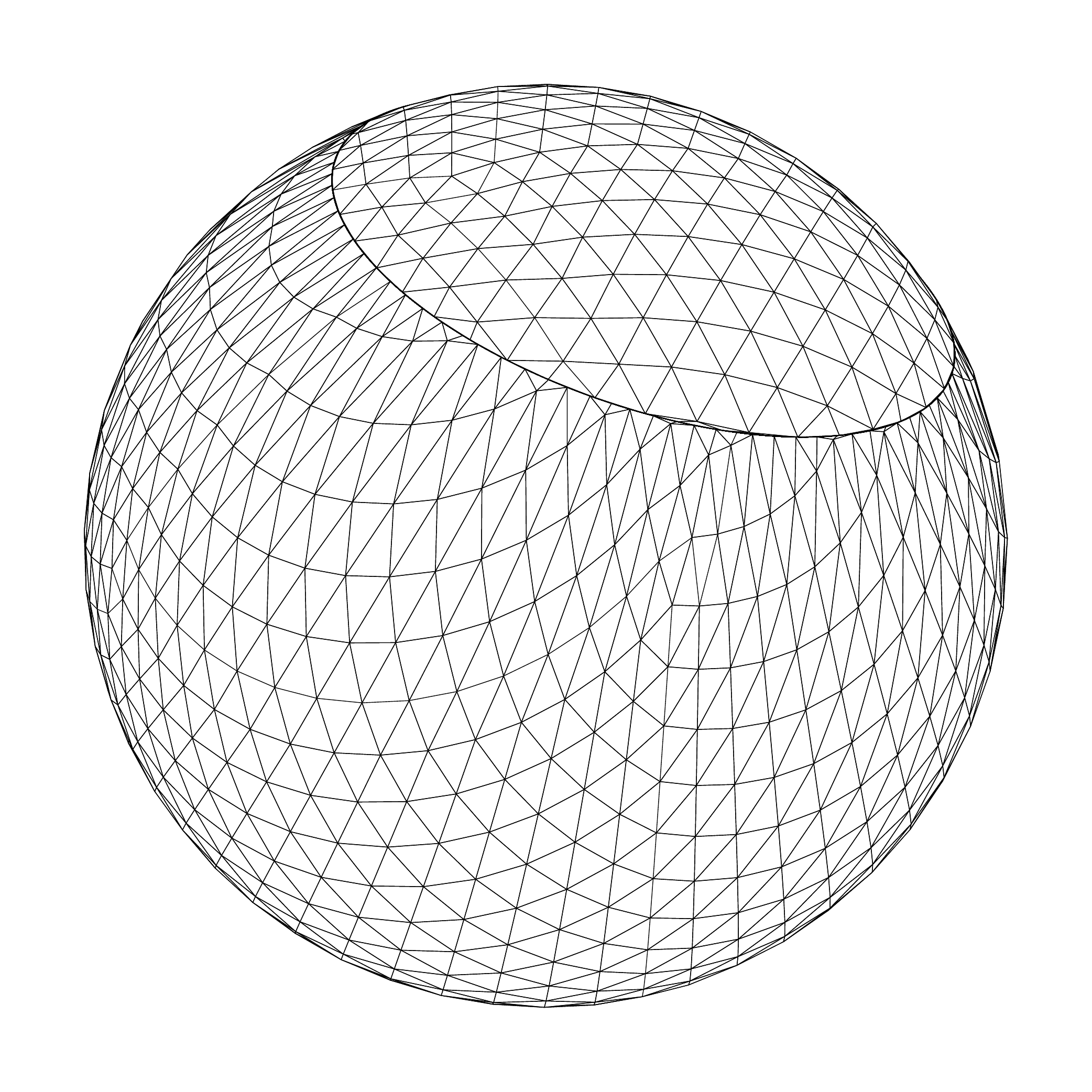}
  \includegraphics[width=0.45\columnwidth]{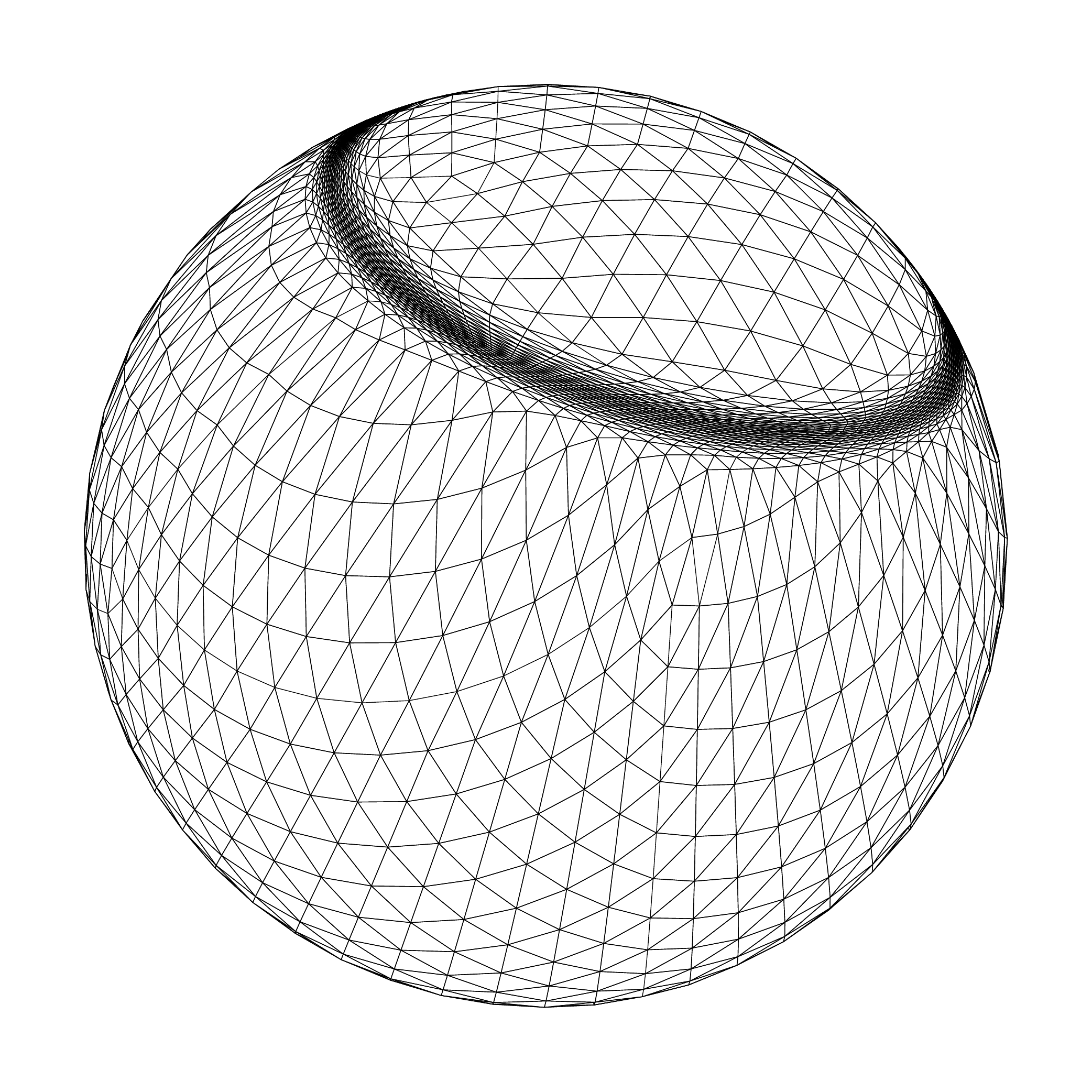}
  \caption{An icosahedral mesh, adapted to the delta-function-based and
sech-function-based ring monitor functions.}
\label{fig:ri-ic-meshes}
\end{figure}

We next show ring meshes induced by these monitor functions.
Cubed-sphere meshes are shown in \cref{fig:ri-cs-meshes}, while
icosahedral meshes are shown in \cref{fig:ri-ic-meshes}. While the
meshes adapted to delta functions are obviously unsuitable for
calculations, outside the singular ring they are remarkably similar to
those adapted to the corresponding smooth sech function. The smoother
meshes are, as predicted, quite skew in the ring region. However,
they are well-aligned with the ring itself, this will lead to lower
interpolation estimates of the underlying solution \citep{huang2011adaptive}.

\subsection{Direct numerical calculations of the skewness of the resulting meshes}
\label{ssec:numaxi4}

Later, in \cref{sec:numnonaxi}, we will apply the general numerical
methods developed in \citet{mcrae2018optimal} to generate examples of
meshes produced by non-axisymmetric monitor functions. We firstly apply
this approach to the axisymmetric examples above so that we can compare
the output of the numerical method to the analytic results of mesh
regularity we have obtained so far.

Given an analytically prescribed monitor function, the numerical method
in \citet{mcrae2018optimal} solves a Monge--Ampère equation of the form
\begin{equation}
\label{eq:mdetstuff}
m(\vec{x}) \det \left((\nabla\exp(\nabla u)\vec{\xi})\cdot P_\xi + \exp(\nabla u)\vec{\xi}\otimes\vec{\xi}\right) = \alpha
\end{equation}
for the scalar mesh potential $u(\vec{\xi})$, where $P_\xi$ is a
projection matrix and $\alpha$ a normalisation constant. The coordinate
mapping $\vec{x}(\vec{\xi})$ can then be derived from $u$ via the
exponential map \cref{eq:expmap}, and lives in the finite element space
$(P_2)^3$. The derivation of \cref{eq:mdetstuff} is given in
\citet{mcrae2018optimal}; briefly, the determinant term is a way to
express the area ratio $r(\vec{\xi})$ in \cref{eq:equi}, under the
assumption that all calculations are done with the sphere immersed in
$\mathbb{R}^3$.

Given the coordinate map $\vec{x}(\vec{\xi})$ obtained by this method,
we can do postprocessing to obtain an approximation to the skewness
quantity $Q$. We define the raw Jacobian $J$ as the $L^2$-projection of
$\nabla_\xi \vec{x}$ into the finite element space
$(P_2)^{3 \times 3}$. Analytically, the column space of $J$ would have
no component normal to the sphere, but this is not true in the presence
of discretisation errors. We therefore form
$J' = (I - \vec{x}\vec{x})\cdot J$ to eliminate this component
completely. At each node, we then perform an SVD: $J' = U\Sigma V$.
The matrix $\Sigma$ is a diagonal matrix containing the singular values
of the map $\sigma_1, \sigma_2$, and a third entry $\sigma_3 \approx 0$.
The first column of $U$, $\vec{u}_1$, is a vector in the direction of
maximum local stretching, while $\vec{u}_2$ is at right angles to this.
The vectors $\sigma_1 \vec{u}_1$ and $\sigma_2 \vec{u}_2$ therefore
represent the local stretching of the mesh.

In the following figures, we show the skewness factor $Q$, together with
the vectors $\sigma_1 \vec{u}_1$ and $\sigma_2 \vec{u}_2$, for several
of the examples considered previously. For brevity, we use an
icosahedral mesh in each case. \Cref{fig:skewness-axi-tanh} shows the
smoother top-hat regional mesh example, and \cref{fig:skewness-axi-sech}
shows the sech-based ring example.

\begin{figure}[!tb]
  \centering
  \includegraphics[width=0.8\columnwidth]{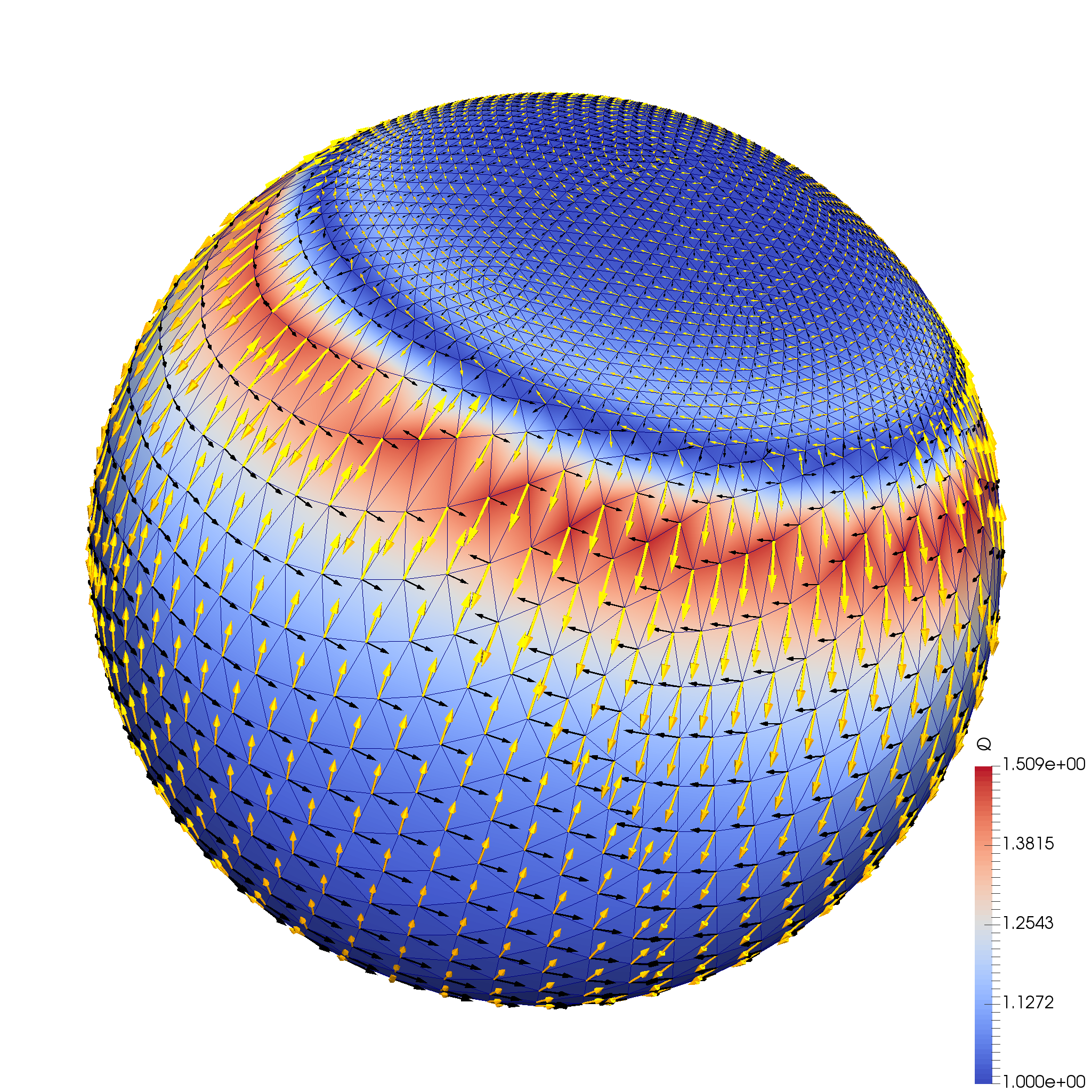}
  \caption{Numerical skewness calculation for the smoother top-hat regional
monitor function. The background colour shows the skewness, $Q$, while
the arrows represent the scaled singular vectors $\sigma_1 \vec{u}_1$
(orange) and $\sigma_2 \vec{u}_2$ (black). The `flipping' of arrows
between neighbouring cells is an artifact of the sign-ambiguity of the
SVD, and is not supposed to be physically significant. It is clear that
$Q$ approaches 1 at the poles $\pm \vec{\omega}$. The maximum skewness
is found just outside the high-resolution region, as predicted by
\cref{fig:Q_tophat}, and the intermediate band of $Q = 1$ is visible. In
the outer region, there is meridional stretching, so the leading
singular vector is in the meridional direction. In the inner region,
there is meridional compression, so the leading singular vector is in
the zonal direction.}
\label{fig:skewness-axi-tanh}
\end{figure}

\begin{figure}[!tb]
  \centering
  \includegraphics[width=0.8\columnwidth]{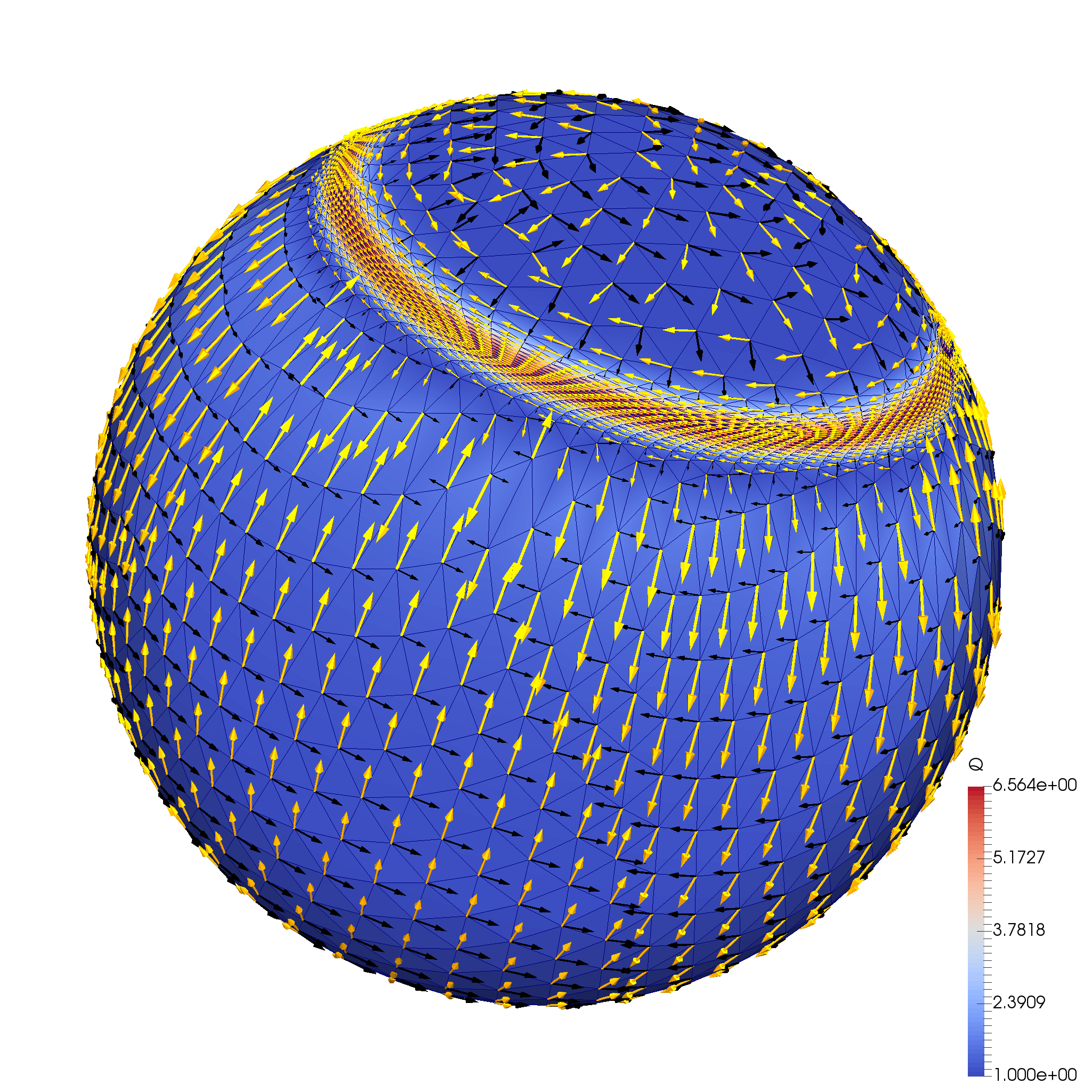}
  \caption{As \cref{fig:skewness-axi-tanh}, for the sech-based ring
monitor function. The maximum skewness is found in the high-resolution
ring, and is close to the earlier prediction of approximately 6.4. The
return to $Q = 1$ is visible as a darker blue band just outside the
ring. Within the ring, there is extreme meridional compression, so the
leading singular vector is in the zonal direction. Outside the ring,
there is mild meridional stretching, so the leading singular vector is
in the meridional direction. Observe as a consequence, the excellent
alignment of the mesh with the ring feature.}
\label{fig:skewness-axi-sech}
\end{figure}

The results of these numerical calculations agree with the earlier
analytical estimates. It can be seen that the vectors
$\vec{u}_1$ and $\vec{u}_2$ generally point towards the pole and
perpendicular to it, a consequence of the axisymmetry of the examples.
This is occasionally violated in the regions where $Q \approx 1$, in
which there is no dominant direction of local stretching. The maximum
skewness is somewhat underestimated in the point singularity example.
This is perhaps not surprising: the large skewness is caused by
stretching of cells in the meridional direction. The meridional
resolution is therefore poor precisely where the skew is large, and
fairly large discretisation errors can be expected. Using a once-refined
mesh (not shown) gives the closer estimate of 5.59.

\section{A comparison with meshes on the plane}
\label{sec:planecomp}

We now briefly compare the regularity of a mesh on the sphere with an
`equivalent' mesh on a subset of the plane, extending the analysis in
\citet{budd2015geometry}. The conclusion of this section is that meshes
on the sphere have, as expected, a higher regularity than those on the
plane.

On the plane, we consider a radially-symmetric
monitor function, which induces a map $R(r)$ from computational space to
physical space. The equidistribution condition then leads to
\begin{equation}
\label{eq:equirad}
m(R) \dd{R}{r} R = \alpha r,
\end{equation}
with $\alpha$ a normalisation constant. This can be compared
with \cref{eq:equiaxi} for the sphere. For a disc of radius $\pi$,
$\alpha$ satisfies
\begin{equation}
\label{eq:radalpha}
\frac{1}{2} \alpha \pi^2 = \int_0^\pi m(R) R \,\mathrm{d}R,
\end{equation}
c.f.\ \cref{eq:axialpha}. In \citet{budd2015geometry}, it is shown that
the eigenvalues of the map are $\mathrm{d}R/\mathrm{d}r$ and $R/r$. By
following similar steps to \cref{ssec:localreg}, the skewness $Q$ is
given by
\begin{equation}
\label{eq:Qrad}
Q = \frac{1}{2} \left( \frac{\alpha}{m(R)} \left( \frac{r}{R} \right)^2 + \frac{m(R)}{\alpha} \left( \frac{R}{r} \right)^2 \right).
\end{equation}

For $r$ and $R$ small, we have $R/r \to dR/dr$ as $r \to 0$. As on the
sphere, it follows that
\begin{equation}
Q \to 1 \quad \mbox{as} \quad r, R \to 0,
\end{equation}
and hence the mesh is very regular in this limit. However, as we
approach the boundary (which we assume is mapped to itself, so that
$R(a) = a$), we have
\begin{equation}
Q = \frac{1}{2} \left( \frac{\alpha}{m(a)} + \frac{m(a)}{\alpha} \right),
\end{equation}
This value is completely arbitrary; there is no control over $Q$ as we
approach the boundary, and the mesh could be very skew there, as was
observed in \citet{budd2015geometry}. This behaviour could not occur on
the sphere due to the absence of a boundary.
To give a specific numerical example, consider the sech-based ring
monitor function \cref{eq:mring},
\begin{equation}
m(\theta') = 1 + \frac{\beta}{\epsilon} \sech^2\left(\frac{\theta'^2 - \Theta'^2}{\epsilon}\right),
\end{equation}
together with an `equivalent' radially-symmetric monitor function for a
disk of radius~$\pi$,
\begin{equation}
m(R) = 1 + \frac{\beta}{\epsilon} \sech^2\left(\frac{R^2 - \Theta'^2}{\epsilon}\right).
\end{equation}
For comparison we take $\Theta' = \pi/4$, $\beta = 5\pi/4$,
and $\epsilon = \pi/50$. The resulting skewness estimate is shown in
\cref{fig:Q_sphpla}. Both cases naturally lead to large skewness in the ring
itself. However, the secondary peak is much larger in the planar case
than for the sphere, corresponding to (unwanted) radial stretching of
cells. This is consistent with the mesh shown in Fig 4.3 of
\citet{budd2015geometry}.

\begin{figure}[!tb]
  \centering
  \includegraphics[width=0.6\columnwidth]{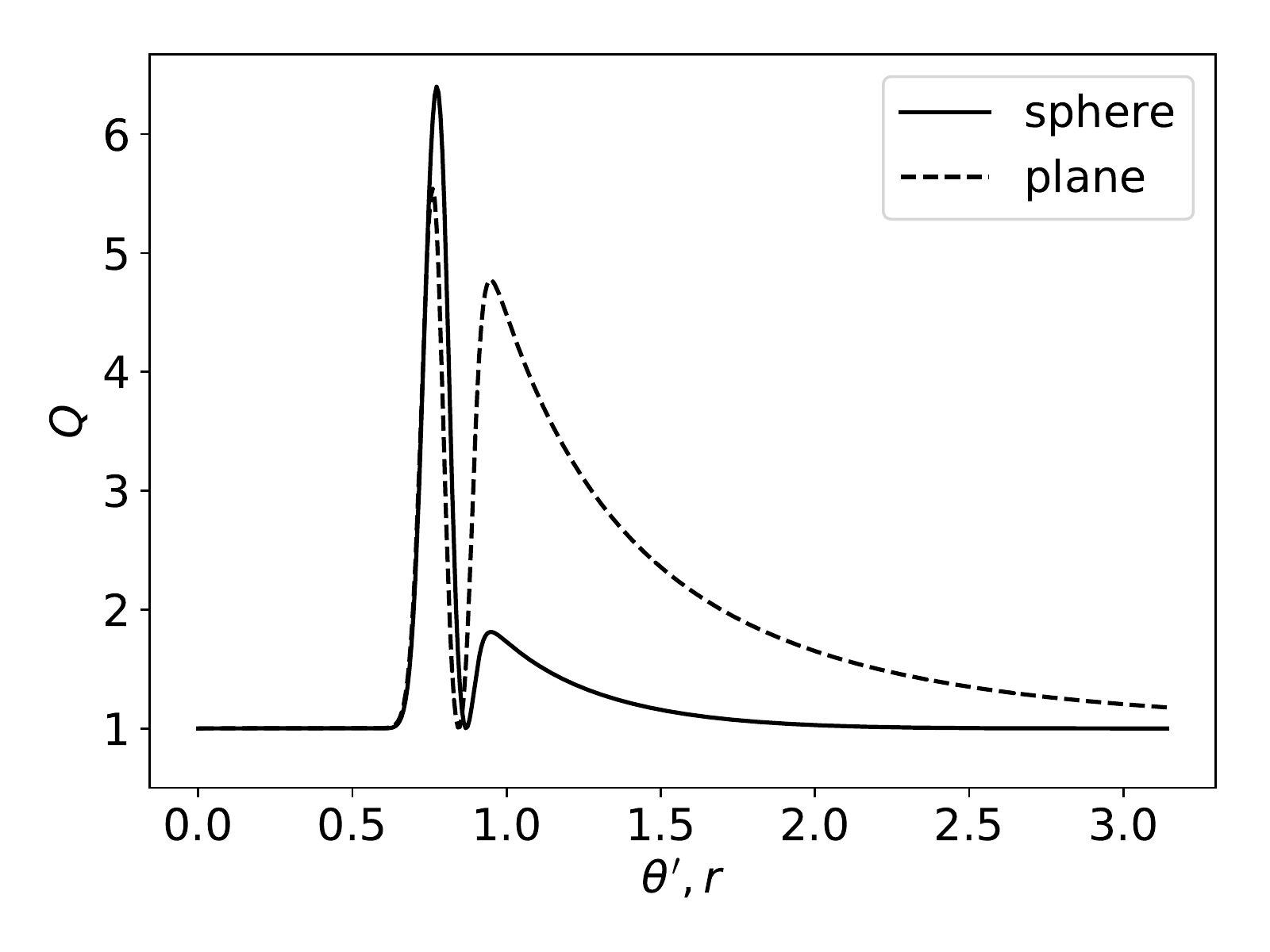}
  \caption{The skewness functions for the maps generated by the
sech-based ring monitor function on the sphere and on the plane. Both
cases lead to large skewness in the ring itself, which is expected, and
is likely to be desirable. However, the differing geometries result in
the plane having a far larger secondary peak than the sphere. This
secondary peak corresponds to cells outside the ring being radially
stretched in order to cover the remaining area.}
\label{fig:Q_sphpla}
\end{figure}

\section{More general meshes and comparison to some other methods}
\label{sec:numnonaxi}

In the previous sections, we considered several examples that made use
of axisymmetric monitor functions to generate meshes. Whilst these
meshes are, in certain cases, computationally useful, and can be
analysed exactly, they do not, of course, have the generality of the
meshes required for calculating the solution of most PDEs.  Without the
axisymmetric restriction, the equidistribution requirement and the
optimal transport condition leads to a generalised Monge--Ampère
equation which can be expressed with respect to spherical angles, as in
\cref{ssec:otcoord}, or with respect to the background Cartesian space
$\mathbb{R}^3$, as was done in \citet{mcrae2018optimal}.

The purpose of this section is to now consider some interesting
non-axisymmetric examples of meshes obtained from more general monitor
functions. This is done by solving the Monge--Ampère-like equation
numerically using the methods described in \citet{mcrae2018optimal}. We
will look at the compression, skewness and alignment of the resulting
meshes as well as demonstrating the robustness and flexibility of the
method. We also present some examples that can be compared with meshes
generated by other methods.

\subsection{Two intersecting rings}

\begin{figure}[!tb]
  \centering
  \includegraphics[width=0.8\columnwidth]{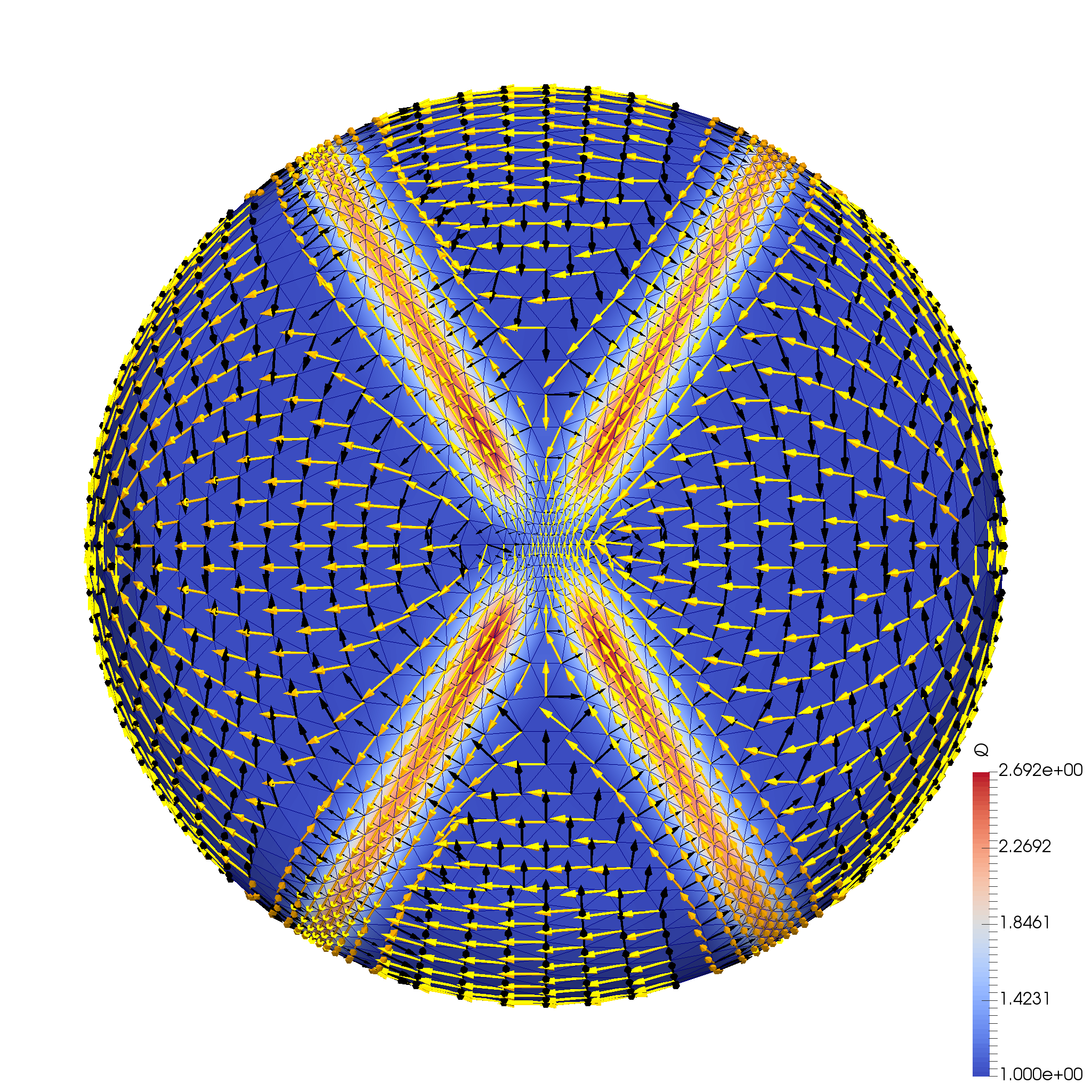}
  \caption{As \cref{fig:skewness-axi-tanh}, for the `cross' monitor
function \cref{eq:m-cross}. The mesh is well-aligned to the intersecting
ring features. Within the rings, the dominant singular vector is in the
direction of the ring. Away from the cross, the mesh undergoes very mild
stretching in order to `provide' resolution to the ring features.}
\label{fig:skewness-axi-cross}
\end{figure}

Our first example is shown in \cref{fig:skewness-axi-cross} and
comprises two intersecting ring features. This is based on (though
different to) the `cross' example in \citet{mcrae2018optimal}.
The monitor function that generates this is given by
\begin{equation}
\label{eq:m-cross}
  m(\vec{x}) = \prod_i^N \left(1 + \alpha_i \sech^2(\beta_i ( \|\vec{x} - \omega_i\|^2  - (\pi/2)^2))\right).
\end{equation}
We take $N$ = 2, $\alpha_1 = \alpha_2 = 5$, $\beta_1 = \beta_2 = 5$,
and the axes $\omega_1, \omega_2 = (\pm \sqrt{3}/2,0,1/2)$ are such that
the rings cross at an angle of 60 degrees. The mesh cells in the ring
are reasonably skew, as expected, and locally the mesh compression and
regularity are very similar to those for the single ring considered in
\cref{sec:numaxi}. The cells outside the rings are almost unaffected and
show great regularity. A key observation from this figure is that at the
point where the rings intersect, and where we might expect to see a very
distorted mesh, we see instead that the skewness is small, and that the mesh cells are very nearly
uniform. 

\subsection{Sinusoidally-varying ring}

\begin{figure}[!tb]
  \centering
  \includegraphics[width=0.8\columnwidth]{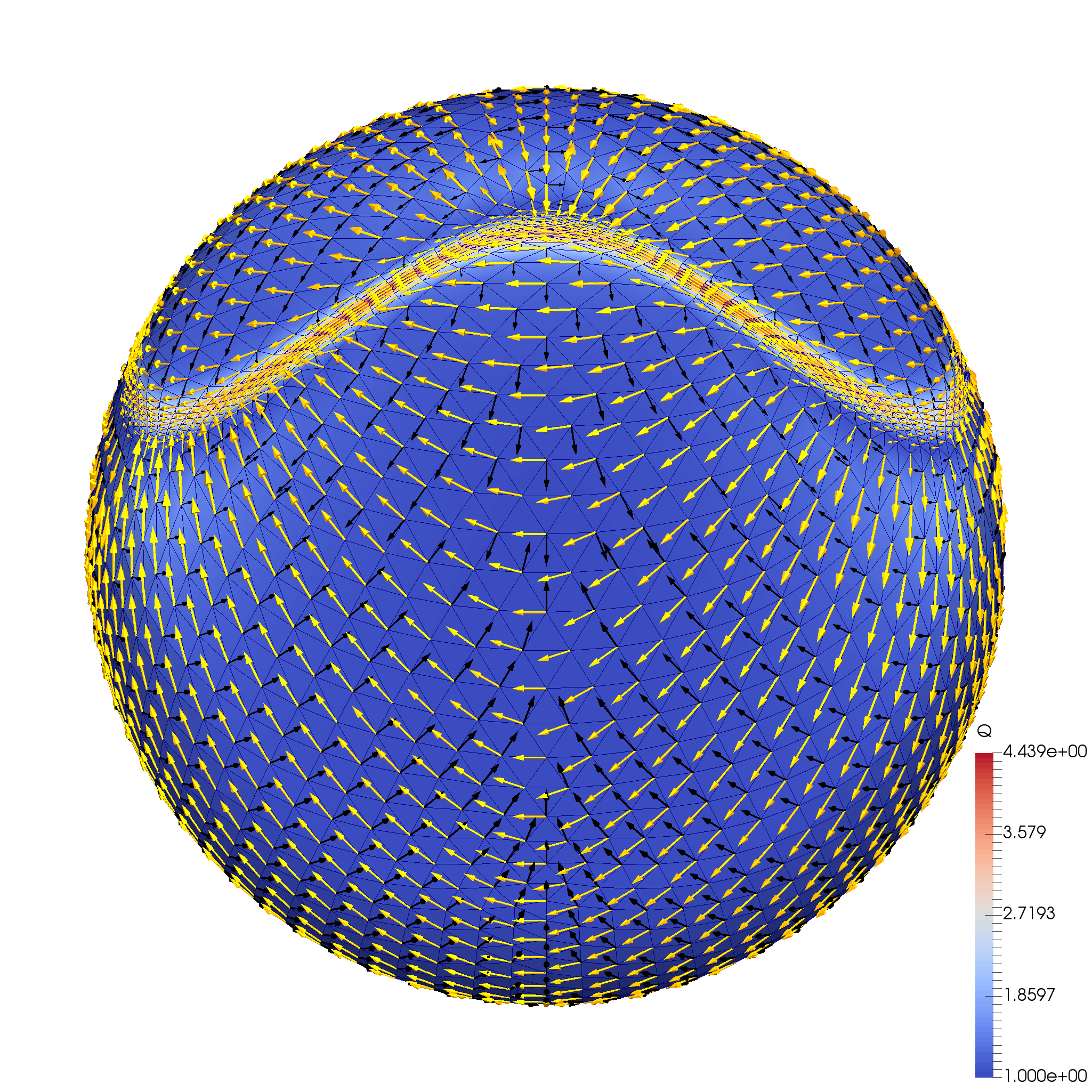}
  \caption{As \cref{fig:skewness-axi-tanh}, for the sinusoidal monitor
function \cref{eq:m-rossby}. The dominant singular vectors show that the
mesh is very well aligned to the sinusoidal feature. There is some mild
stretching of the mesh immediately outside the feature in order to
provide the enhanced resolution. Away from the feature, the mesh is
incredibly regular.}
\label{fig:skewness-axi-rossby}
\end{figure}

Our second example is shown in \cref{fig:skewness-axi-rossby}. This mesh
is induced by a monitor function which concentrates cells within a
sinusoidal pattern in the northern hemisphere. It is inspired by the
related planar example in \citet{budd2015geometry}, and is not
dissimilar to a Rossby--Hauritz wave (see, for example, test case 6 from
\citet{williamson1992standard}). The monitor function is given by
\begin{equation}
\label{eq:m-rossby}
  m(\vec{x}) = 1 + \alpha\sech(\beta\theta'),\qquad
  \theta' = \theta - (\theta_c + \frac{1}{2}\theta_a\sin(k\phi)),
\end{equation}
where $\theta$ is the latitude (now measured from the equator), $\phi$
is the longitude, $\alpha = 15$, $\beta = 25$, $\theta_c = \frac{\pi}{6}$,
$\theta_a = \frac{\pi}{6}$, and we use a wavenumber $k = 3$. The mesh
cells are well-aligned with the sinusoidal pattern. There is also only
slight stretching of cells outside this high-resolution region and the
grid is very regular there. Hence this mesh would be very suitable for
computing a Rossby-Hauritz wave with good resolution. The local
properties of this mesh close to the wave are very similar to that of
the ring examples considered earlier.

\subsection{Equatorially-enhanced mesh}

Inspired by \citet{iga2017equatorially}, our next example is a mesh that
concentrates points in an equatorial region. That paper uses a
specially-constructed triangular mesh with an elaborate topology that
places far more cells around the equator than a normal icosahedral mesh.
A spring dynamics approach is then used to smooth the mesh, followed by
analytical transformations around problematic points. Here, we use the
optimal transport procedure to generate a non-uniform mesh from a
nearly-uniform icosahedral mesh, requesting a similar distribution of
resolution.

We define a target grid spacing $d^*$, which is a function of latitude 
$\theta'$ only (expressed in degrees). This takes the values
\begin{equation}
d^*(\theta') =
\begin{cases}
0.064, &|\theta'| < 13\\
0.064 + \frac{|\theta'| - 13}{31 - 13} (0.23 - 0.064), &13 \leq |\theta'| \leq 31\\
0.23, &|\theta'| > 31.
\end{cases}
\end{equation}
This is a similar target grid spacing to the `analytic resolution
distribution' in Figure 8a of \citet{iga2017equatorially} (note that
this is \emph{derived from their mesh topology}, rather than being
specified in advance). We then use a monitor function $m \propto 1/d^{*2}$
to control the cell area in the solution of the Monge--Ampère equation.
The mesh generation technique automatically obtains the correct constant
of proportionality.

\begin{figure}[!tb]
  \centering
  \includegraphics[width=0.7\columnwidth]{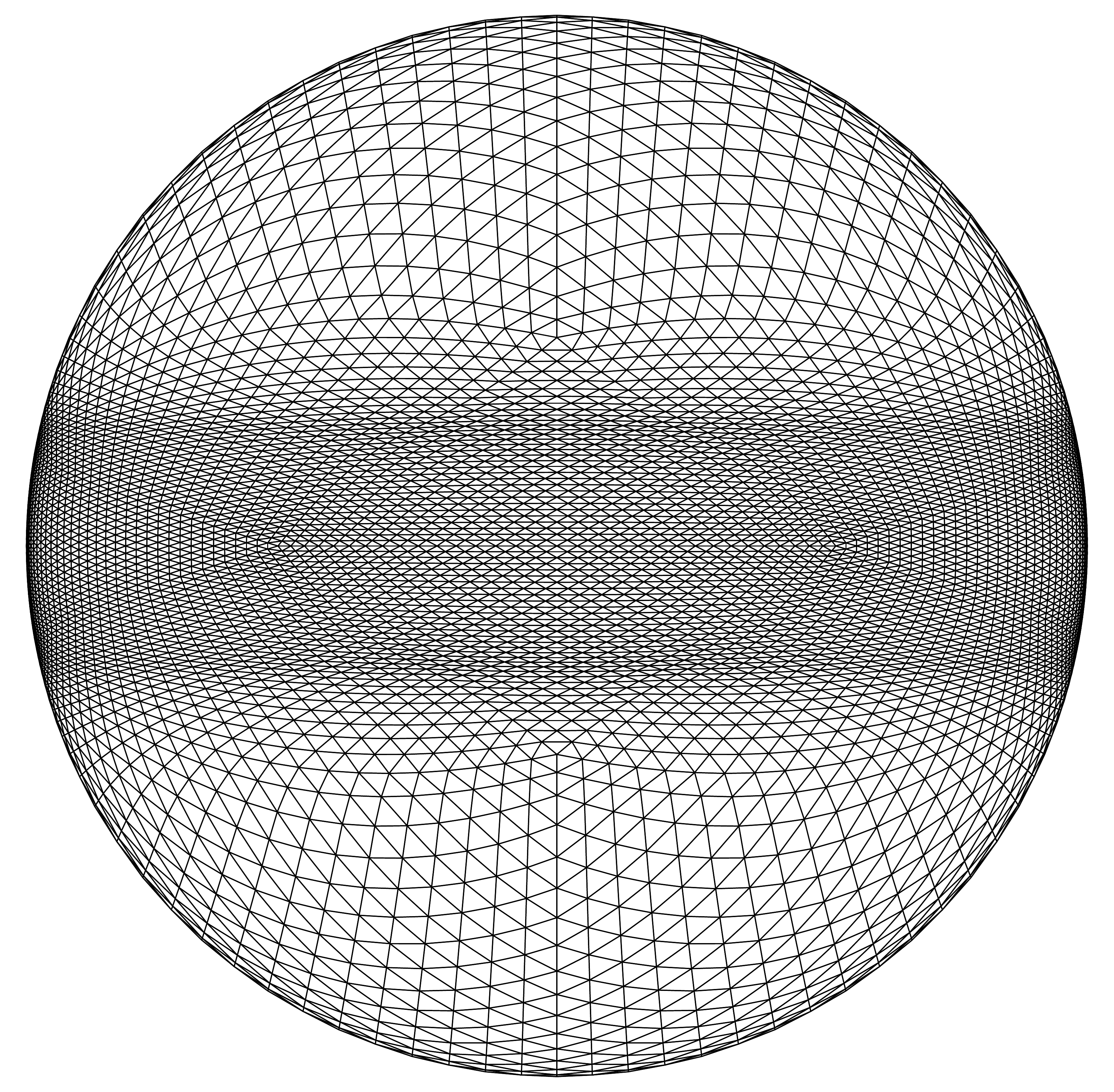}
  \caption{An equatorially-enhanced mesh produced using the optimal
transport approach.}
\label{fig:iga-mesh}
\end{figure}

\begin{figure}[!tb]
  \centering
  \includegraphics[width=0.58\columnwidth]{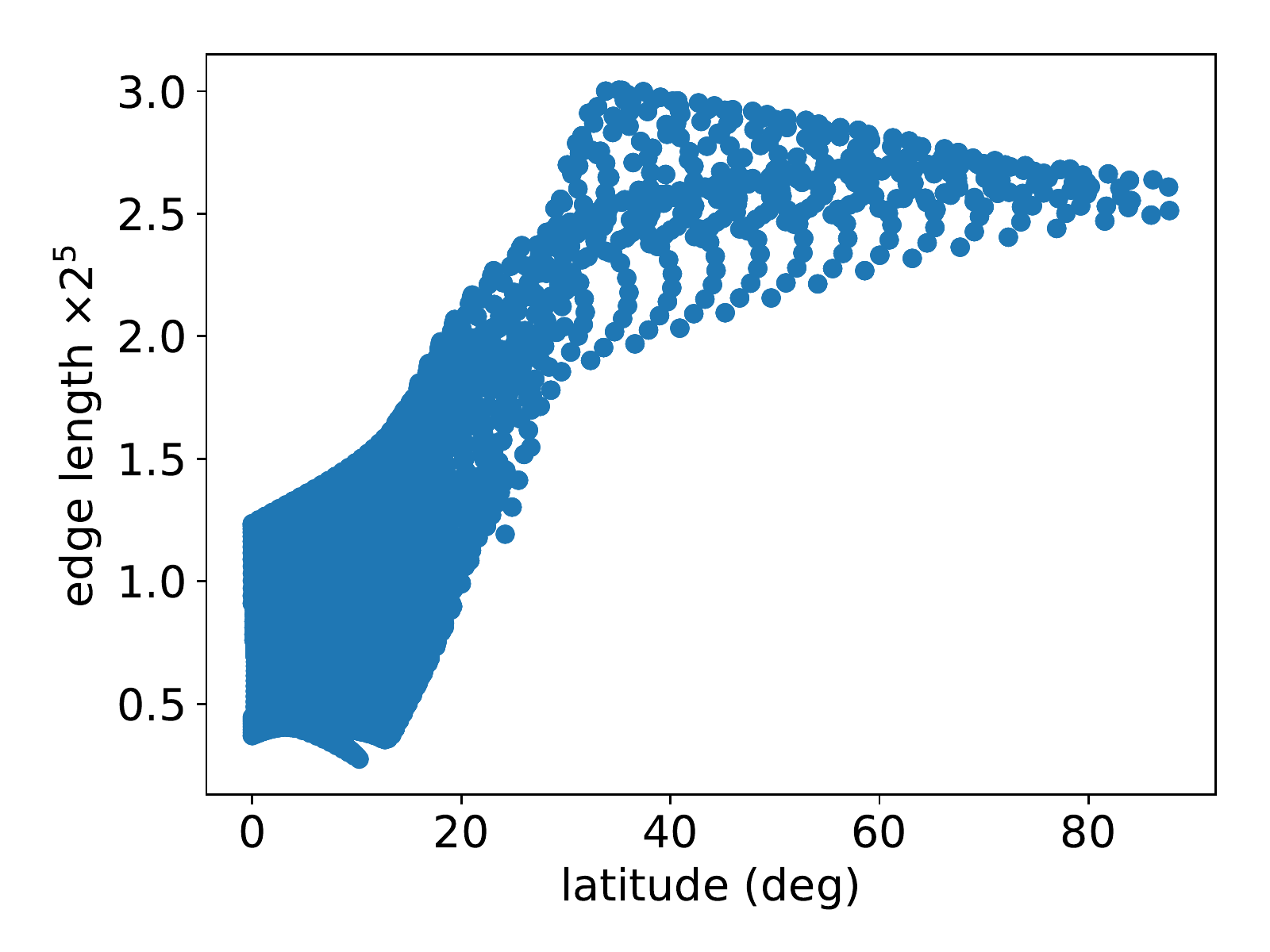}
  \caption{A graph of normalized edge length against latitude for the
equatorially-enhanced mesh (for symmetry reasons, only the northern
hemisphere is shown).}
\label{fig:iga-comp}
\end{figure}

The resulting mesh is shown in \cref{fig:iga-mesh}, and a graph of
mesh spacing against latitude is given in \cref{fig:iga-comp}. It is
clear from both figures that the desired mesh compression has been
achieved, and that the resulting mesh is smooth and has good regularity.
In comparison to the meshes in
\citet{iga2017equatorially}, our technique leads to slightly more
stretching of cells around the equator. This is unavoidable, as our
topology is fixed, and the only way to reduce cell area is to relocate
cells towards the equator from the mid-latitudes.

\subsection{A more-uniform icosahedral mesh}

Our final example uses optimal transport to tackle the minor
nonuniformities in a refined icosahedral mesh. The standard approach to
generating a refined icosahedral mesh is to refine the faces of an
icosahedron, then to project mesh vertices radially outwards onto the
surface of the sphere. The resulting mesh is reasonably uniform, but the
ratio of maximum to minimum cell area is approximately 2. Using our
optimal transport approach, we can equalise the cell areas.

This problem is slightly different to what we have done previously:
compared to \cref{eq:equi}, we instead have $m(\vec{\xi}) r(\vec{\xi}) = \alpha$.
The monitor function now depends on the computational coordinate
$\vec{\xi}$ rather than the physical coordinate $\vec{x}$. The
computation of the map is slightly easier, as a result, since one source
of nonlinearity is removed from the resulting PDE. We represent $m$ as a
piecewise-constant function defined by the areas of each cell on the
unadjusted mesh.

\begin{figure}
  \centering
  \includegraphics[width=0.6\columnwidth]{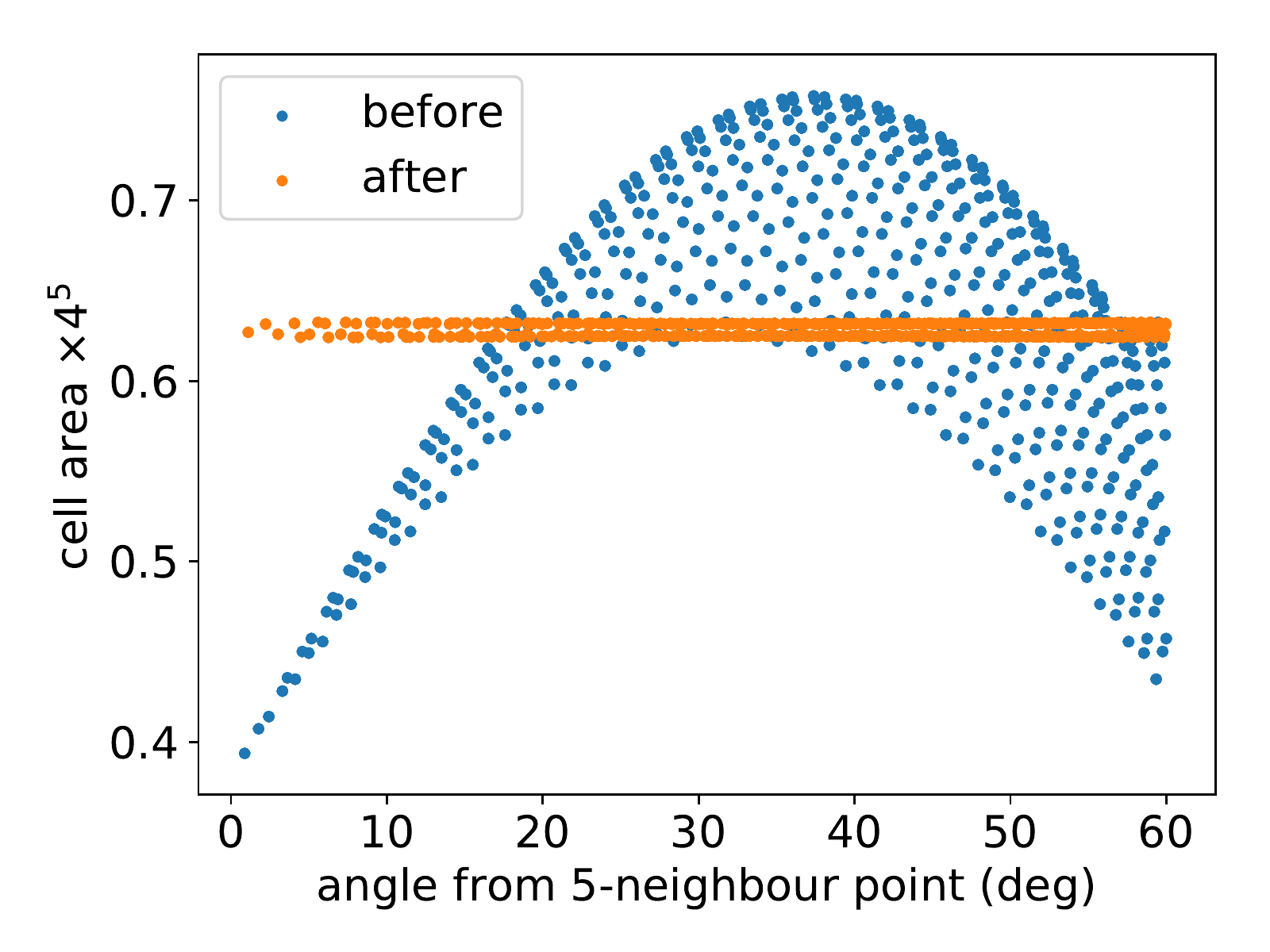}
  \caption{A graph of cell area against angle from the special
5-neighbour points, before and after the optimal transport procedure.
Before adjustment, the smallest cells are gathered around these special
points, and the max-min area ratio is approximately 1.925. After
adjustment, the max-min area ratio is just 1.013.}
\label{fig:icos-areas}
\end{figure}

\begin{figure}
  \centering
  \includegraphics[width=0.6\columnwidth]{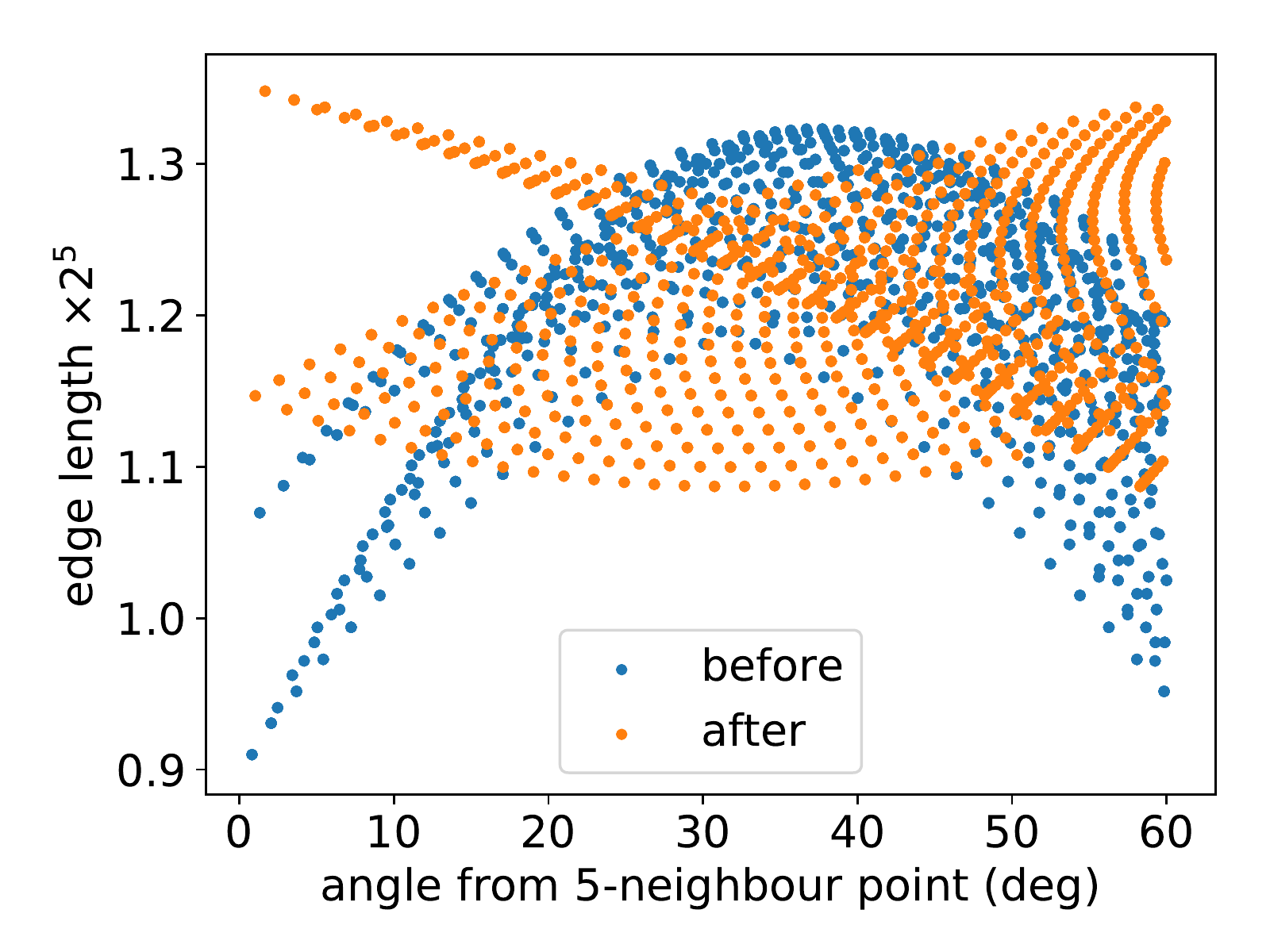}
  \caption{A graph of normalised edge length against angle from the
special 5-neighbour points, before and after the optimal transport
procedure. Globally, the edge lengths become more uniform, with a
resulting variation of approximately 20\%.}
\label{fig:icos-edgelengths}
\end{figure}

\begin{figure}[!tb]
  \centering
  \includegraphics[width=0.6\columnwidth]{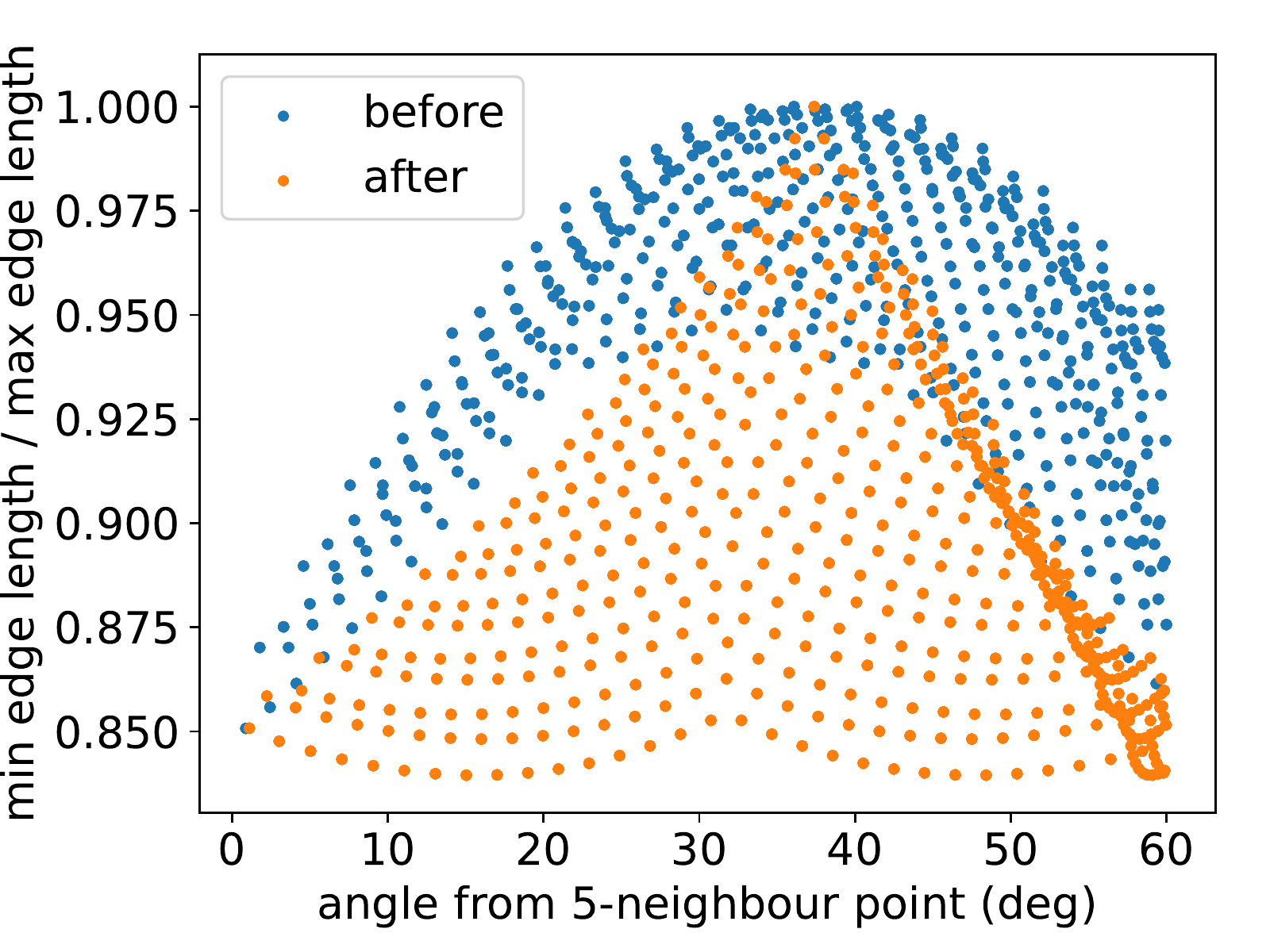}
  \caption{A graph of cell regularity, defined as the ratio of minimum
to maximum edge length, against angle from the special 5-neighbour
points, before and after the optimal transport procedure. While the
total range has hardly changed, the equal-area mesh has slightly less
regular cells than the unadjusted mesh.}
\label{fig:icos-maxmin}
\end{figure}

We give some graphs analysing the effect of this procedure. A graph of
cell area against angle from a specific 5-neighbour point is shown in
\cref{fig:icos-areas}; this can be compared with Figures 1 and 13 in
\citet{iga2014improved}. The resulting cell areas still vary by about 1\%,
due to discretisation error, but this compares very favourably with
other methods. The optimal transport approach controls the area scaling,
so it is natural that we have strong control over the area of the
resulting mesh cells. In \cref{fig:icos-edgelengths}, we plot the
normalised edge lengths against angle. Globally, there is less variation
in the edge lengths after the optimal transport procedure, which is
consistent with the equalisation of cell areas. In \cref{fig:icos-maxmin},
we plot a discrete measure of mesh regularity. This is defined (for ease
of comparison with other methods) as the minimum-to-maximum edge ratio
for each cell. The adjustment procedure leads to cells which are
generally slightly less regular (equilateral) than the original.
Overall, we see that the optimal transport approach has acted as an
effective `mesh smoother', giving a more uniform mesh than the original.
The adjusted mesh is thus very suitable for computing solutions of
certain PDEs.

\section{Conclusions}
\label{sec:conc}

In this paper, we have described a flexible method that produces dynamic
adapted meshes on the sphere which are topologically identical to a
given logical mesh. This involves calculating a map from the sphere
to itself, by solving a PDE of Monge--Ampère type, and applying this map
to the input mesh. By using an optimal transport strategy to find this
map, we have a robust method of constructing such a dynamic mesh.
Significantly, a-priori mesh regularity properties are inherited from
the regularity of this optimal transport map. In particular, we show,
both theoretically and by example, that such meshes on the sphere are
more regular than similar meshes on the plane.

By looking at a particular class of solutions, we analytically derived
some specific maps and also their associated skewness properties. Hence
we could generate various meshes, with provable regularity estimates,
which could be used for practical computations. We also considered more
general examples, calculating the maps and meshes numerically, and
showed that the resulting meshes still have good regularity properties
and compare favourably to those given by other methods. We
observe that the meshes generated through the optimal transport approach
are capable of producing anisotropic meshes aligned with the solution
features, even though the monitor function is only scalar-valued. In
another example we have demonstrated that optimal transport is effective
as a `mesh smoother' for increasing the uniformity of a mesh.

In this paper, we only considered the case where the domain is the
entire sphere $S^2$, which has no boundary. Some applications may
require a mesh for only a subset of the sphere, particularly a
geophysical simulation such as an ocean on the earth. In Euclidean
space, optimal transport theory requires convex domains, which would
seem to rule out this sort of application. Furthermore,
\citet{mccann2001polar} only considers boundary-free manifolds. However,
it is sometimes possible to bypass the convexity requirement by working
on an extended, convex, domain, and setting the monitor to zero outside
the true domain. This is done in, for example,
\citet{benamou2014numerical}. We have not attempted to replicate this
on the sphere.

A realistic geophysical application on the sphere would likely require
a three-dimensional mesh. The earlier paper \citet{browne2014fast}
already considered fully three-dimensional optimal transport mesh
adaptivity in a cuboid domain. Unfortunately, the analogous spherical
shell is not a convex domain, so our method would not work without
further modifications. However, it is unlikely that full three-dimensional
adaptivity is desirable in a geophysical application, since accurate
representation of pressure gradient terms necessitates that cells should
be in vertically-aligned columns. It is more likely that some kind of
2+1D adaptivity would be used. The base mesh can be adapted following
the methods described in this paper, then the nodes in each column can
be relocated up or down separately. This also reduces the computational
complexity of the problem significantly.

Of course, for useful problems, finding a mesh is just one part of
solving a physical problem represented by the time-evolving solution of
a PDE. The equation must then be discretised on the mesh using, for
example, a finite volume or finite element approach. There are numerous
issues that must still be investigated. In the context of
advection-dominated flows, this includes the accuracy and stability of
the solution, the dispersion relations of any waves, the ability of the
mesh to support balanced flows, and the construction of suitable monitor
functions to achieve these. This ongoing research will be the subject of
future papers.

\section*{Acknowledgements}
We would like to thank Hilary Weller and Jemma Shipton for useful
discussions about mesh properties, and William Saunders for help in
producing the vector graphics. We would also like to thank the anonymous
referees for their very helpful comments on an earlier version of this
paper. This work was supported by the Natural Environment Research
Council [grant numbers NE/M013480/1, NE/M013634/1]. This project has
received funding from the European Research Council (ERC) under the
European Union's Horizon 2020 research and innovation programme (grant
agreement no 741112).

\appendix

\section{Code availability}

The numerical calculations made use of \emph{SciPy} \citep{scipy-cite},
particularly the integration and optimisation routines. When running the
numerical mesh-generation methods developed in \citet{mcrae2018optimal},
we used the \emph{Firedrake} finite element software
\citep{rathgeber2016firedrake}, including specialist functionality
developed in \citet{rognes2013automating,alnaes2014unified,
mcrae2016automated,homolya2016parallel,luporini2017algorithm,
homolya2018tsfc,homolya2017exposing}. Firedrake itself relies on
\emph{PETSc} \citep{petsc-user-ref,petsc-efficient} and \emph{petsc4py}
\citep{dalcin2011parallel}.

The code for the numerical experiments can be found in the supplementary
material to this paper. While the code should be compatible with
Firedrake for the foreseeable future, the precise versions of Firedrake
components that were used in this paper are archived at
\citet{zenodo_firedrake}.

\bibliography{spheregeom}

\end{document}